\documentclass[10pt]{amsart}

\usepackage{enumerate}
\usepackage{comment}

\makeatletter

\@namedef{subjclassname@2010}{

  \textup{2010} Mathematics Subject Classification}

\usepackage{amssymb, amsmath,amsthm}
\usepackage{mathrsfs}
\newtheorem{thm}{Theorem}[section]
\newtheorem*{thm1}{Theorem}
\newtheorem{prop}[thm]{Proposition}
\newtheorem{conj}[thm]{Conjecture}

\newtheorem{cor}[thm]{Corollary}
\newtheorem{lem}[thm]{Lemma}
\theoremstyle{definition}
\newtheorem{rem}[thm]{Remark}
\newtheorem{def1}[thm]{Definition}

\newcommand{\ra}{\rightarrow}
\newcommand{\bk}{\backslash}
\newcommand{\mc}{\mathcal}

\newcommand{\mb}{\mathbb}
\newcommand{\sg}{\sigma}
\newcommand{\eps}{\epsilon}

\renewcommand{\ss}{\substack}
\newcommand{\llf}{\left\lfloor}

\newcommand{\e}{\varepsilon}
\newcommand{\rrf}{\right\rfloor}

\newcommand{\mbf}{\boldsymbol}
\newcommand{\asum}{\sideset{}{^{\ast}}\sum}
\renewcommand{\bar}{\overline}

\begin{document}
\title{Additive Functions in Short Intervals, Gaps and a Conjecture of Erd\H{o}s}
\author{Alexander P. Mangerel}
\address{Centre de Recherches Math\'{e}matiques, Universit\'{e} de Montr\'{e}al, Montr\'{e}al, Qu\'{e}bec}
\email{smangerel@gmail.com}

\begin{abstract}
With the aim of treating the local behaviour of additive functions, we develop analogues of the Matom\"{a}ki-Radziwi\l\l{} theorem that allow us to approximate  the average of a general additive function over a typical short interval in terms of a corresponding long average. As part of this treatment, we use a variant of the Matom\"{a}ki-Radziwi\l\l \ theorem for divisor-bounded multiplicative functions recently proven in \cite{ManMR}.\\
We consider two sets of applications of these methods. Our first application shows that for an additive function $g: \mb{N} \ra \mb{C}$ any non-trivial savings in the size of the average gap $|g(n)-g(n-1)|$ implies that $g$ must have a small first centred moment, i.e., the discrepancy of $g(n)$ from its mean is small on average. We also obtain a variant of such a result for the second moment of the gaps. This complements results of Elliott and of Hildebrand.\\
As a second application, we make partial progress on an old question of Erd\H{o}s relating to characterizing constant multiples of $\log n$ as the only \emph{almost everywhere} increasing additive functions. We show that if an additive function is almost everywhere non-decreasing then it is almost everywhere \emph{well-approximated} by a constant times a logarithm. We also show that if the set $\{n \in \mb{N} : g(n) < g(n-1)\}$ is sufficiently sparse, and if $g$ is not extremely large too often on the primes (in a precise sense), then $g$ is identically equal to a constant times a logarithm.
\end{abstract}

\maketitle
\section{Introduction}
An arithmetic function $g: \mb{N} \ra \mb{C}$ is called \emph{additive} if, whenever $n,m \in \mb{N}$ are coprime, $g(nm) = g(n)+g(m)$; it is said to be \emph{completely additive} if the coprimality condition on $n,m$ can be ignored. Additive functions are objects of classical study in analytic and probabilistic number theory, given their close relationship with the theory of random walks. 

Much is understood about the \emph{global} behaviour of general additive functions. For instance, the orders of magnitude of all of the centred moments 
$$
\frac{1}{X}\sum_{n \leq X} \left|g(n)-\frac{1}{X}\sum_{n \leq X} g(n)\right|^k, \quad k >0
$$
have been computed by Hildebrand \cite{HilMom}. When $k = 2$, the slightly weaker but generally sharp Tur\'{a}n-Kubilius inequality gives an upper bound, uniform in $g$, of the form
\begin{equation}\label{eq:TKIntro}
\frac{1}{X}\sum_{n \leq X} \left|g(n)-\frac{1}{X}\sum_{n \leq X} g(n)\right|^2 \ll B_g(X)^2,
\end{equation}
where we have denoted by $B_g(X)$ the approximate variance
$$
B_g(X) := \left(\sum_{p^k \leq X} \frac{|g(p^k)|^2}{p^k}\right)^{1/2}.
$$
When $g$ is real-valued one can determine necessary and sufficient conditions according to which the distribution functions $F_X(z) := \frac{1}{X}|\{n \leq X : g(n) \leq z\}|$ converge at continuity points $z \in \mb{R}$; this is the content of the Erd\H{o}s-Wintner theorem. Under certain conditions the corresponding distribution functions (with suitable normalizations) converge to a Gaussian, a fundamental result of Erd\H{o}s and Kac.

Much less is understood regarding the \emph{local} behaviour of additive functions, i.e., the simultaneous behaviour of $g$ at neighbouring integers. Questions of interest from this perspective include: 
\begin{enumerate}[(i)]
\item the distribution of $\{g(n)\}_n$ in \emph{typical} short intervals $[x,x+H]$, where $H = H(X)$ grows slowly,
\item the distribution of the sequence of gaps $|g(n)-g(n-1)|$ between consecutive values, and
\item the distribution of tuples $(g(n+1),\ldots,g(n+k))$, for $k \geq 2$.
\end{enumerate}
Pervasive within this scope are questions surrounding the characterization of those additive functions $g$ whose local behaviour is rigid in some sense; such questions are discussed below, in Section \ref{subsec:apps}. \\
The purpose of this paper is to consider questions of a local nature about general additive functions.
\subsection{Matom\"{a}ki-Radziwi\l\l \ type theorems for additive functions}
The study of additive functions is intimately connected with that of multiplicative functions, i.e., arithmetic functions $f: \mb{N} \ra \mb{C}$ such that $f(nm) = f(n)f(m)$ whenever $(n,m) = 1$.
The mean-value theory of bounded multiplicative functions, which provides tools for the analysis of the global behaviour of multiplicative functions, was developed in the '60s and '70s in the seminal works of Wirsing \cite{WirMV} and Hal\'{a}sz \cite{Hal}. \\
In contrast, the study of the local behaviour of multiplicative functions has long been the source of intractable problems. An important example of this is Chowla's conjecture \cite{Cho}, which generalizes the prime number theorem. This conjecture states, among other things, that for any $k \geq 2$ and any tuple $\mbf{\eps} \in \{-1,+1\}^k$, the set
$$
\{n \leq X: \lambda(n+1) = \eps_1,\ldots,\lambda(n+k) = \eps_k\}
$$
has $(2^{-k}+o(1)) X$ elements, where $\lambda$ is the Liouville\footnote{The Liouville function is the multiplicative function defined as $\lambda(n) := (-1)^{\Omega(n)}$, where $\Omega(n)$ is the number of prime factors of $n$, counted with multiplicity.} function. In other terms, the sequence  of tuples $(\lambda(n+1),\ldots,\lambda(n+k))$ equidistributes among the tuples of signs in $\{-1,+1\}^k$. 

Problems of this type have recently garnered significant interest, thanks crucially to the celebrated theorems of Matom\"{a}ki and Radziwi\l\l \ \cite{MR}. Broadly speaking, their results show that averages of a bounded multiplicative function in typical short intervals are well-approximated by a corresponding long average. In a strong sense, this suggests that the local behaviour of many multiplicative functions is determined by their global behaviour. The simplest version to state is as follows.
\begin{thm1}[Matom\"{a}ki-Radziwi\l\l \ \cite{MR}]
Let $f: \mb{N} \ra [-1,1]$ be multiplicative. Let $10 \leq h \leq X/100$. Then
$$
\frac{2}{X}\sum_{X/2< n \leq X} \left|\frac{1}{h} \sum_{n-h < m \leq n} f(m) - \frac{2}{X}\sum_{X/2 < m \leq X} f(m)\right|^2 \ll \frac{\log\log h}{\log h} + (\log X)^{-1/50}.
$$
\end{thm1}
This result, its natural extensions to complex-valued functions \cite{MRT}, and further improvements, extensions and variants (e.g., \cite{MRII}) have had profound impacts not only in analytic number theory, but equally in combinatorics and dynamics. For instance, Tao \cite{Tao} used this result to develop technology in order to obtain estimates for the logarithmically-averaged binary correlation sums
$$
\frac{1}{\log X}\sum_{n \leq X} \frac{f(n)f(n+h)}{n}, \text{ for multiplicative functions } f: \mb{N} \ra \mb{C}, |f(n)| \leq 1.
$$
This was essential in his proof of the Erd\H{o}s discrepancy problem \cite{TaoEDP}, and also enabled him to obtain a logarithmic density analogue of the case $k = 2$ of Chowla's conjecture. It has also been pivotal in the various developments towards Sarnak's conjecture on the disjointness of the Liouville function from zero entropy dynamical systems (see \cite{SarSur} for a survey).\\
Our first main result establishes an $\ell^1$-averaged comparison theorem for short and long averages of additive functions, inspired by the theorem of Matom\"{a}ki and Radziwi\l\l.
\begin{thm}\label{thm:MRL1}
Let $g: \mb{N} \ra \mb{C}$ be an additive function. Let $10 \leq h \leq X/100$ be an integer\footnote{The requirement that $h$ be an integer is possibly unnecessary, but assuming it allows us to avoid certain pathologies associated with functions $g$ taking very large values.}. Then
$$
\frac{2}{X}\sum_{X/2 < n \leq X} \left|\frac{1}{h} \sum_{n-h < m \leq n} g(m) - \frac{2}{X}\sum_{X/2 < m \leq X} g(m) \right| \ll \left(\sqrt{\frac{\log\log h}{\log h}} + (\log X)^{-1/800}\right)B_g(X).
$$
\end{thm}
\begin{rem}
Theorem \ref{thm:MRL1} should be compared to the ``trivial bound'' arising from applying the triangle inequality, the Cauchy-Schwarz inequality and \eqref{eq:TKIntro} (which is valid for dyadic long averages as well) to obtain
\begin{align*}
&\frac{2}{X}\sum_{X/2 < n \leq X} \left|\frac{1}{h} \sum_{n-h < m \leq n} g(m) - \frac{2}{X}\sum_{X/2 < m \leq X} g(m) \right| \\
&\leq \frac{2}{X}\sum_{X/2-h < m \leq X} \left|g(m) - \frac{2}{X}\sum_{X/2 < n \leq X} g(n)\right| \leq \left(1+\frac{2h}{X}\right)^{1/2}\left(\frac{2}{X}\sum_{X/2-h < m \leq X}\left|g(m)-\frac{2}{X}\sum_{X/2 < n \leq X} g(n)\right|^2\right)^{\frac{1}{2}} \\
&\ll B_g(X).
\end{align*}
Thus, Theorem \ref{thm:MRL1} obtains an improvement over this trivial bound that tends to 0 whenever $h \ra \infty$ as $X \ra \infty$.
\end{rem}
To get a more precise additive function analogue of the Matom\"{a}ki-Radziwi\l\l \ theorem, one would hope to obtain a mean-square (or $\ell^2$) version of Theorem \ref{thm:MRL1}. We are limited in this matter by the possibility of very large values of $g$. Specifically, if $g(p)/B_g(X)$ can get very large for many primes $p \leq X$, it is possible for the $\ell^2$ average to be dominated by a sparse set (i.e., the multiples of these $p$), wherein the discrepancy between the long and short sums is not small. We will thus work with a specific collection of additive functions in order to preclude such pathological behaviour.\\
To describe this collection we introduce the following notations. Given $\e  > 0$ and an additive function $g$, we define\footnote{It is obvious that $\sum_{\ss{p \leq X \\ |g(p)| > \e^{-1}B_g(X)}} p^{-1} \ll \e^2$, and thus the proportion of integers divisible by such a prime is sparse, namely of size $O(\e^2 X)$. Nevertheless, if $F_g(\e) \gg 1$ for all $\e > 0$ the values $g(n)^2$ at multiples of the primes with $|g(p)| > \e^{-1} B_g(X)$ can have an outsized influence on the second moment.}
$$
F_g(\e) := \limsup_{X \ra \infty} \frac{1}{B_g(X)^2} \sum_{\substack{p \leq X \\ |g(p)| > \e^{-1} B_g(X)}} \frac{|g(p)|^2}{p}.
$$
Roughly speaking, $F_g(\e)$ measures the contribution to $B_g(X)^2$ from prime values $g(p)$ of very large absolute value. \\
Clearly, $0 \leq F_g(\e) \leq 1$ for all $\e > 0$ and additive functions $g$. We will concern ourselves with functions $g$ such that $F_g(\e) \ra 0$ as $\e \ra 0^+$, a condition that is satisfied by many additive functions. When $g$ is bounded on the primes e.g., when $g(n) = \Omega(n)$, the number of prime factors of $n$ counted with multiplicity, it is clear that $F_g(\e) = 0$ whenever $\e$ is sufficiently small. For a different example, taking $g = c\log$ for some $c \in \mb{C}$ we find $B_g(X) \sim \frac{|c|}{\sqrt{2}}\log X$, so that $|g(p)| \leq (\sqrt{2} + o(1)) B_g(X)$ for all primes $p$ and hence $F_g(\e) = 0$ for all $\e < 1/2$, say. 
\begin{def1}
We define the collection $\mc{A}$ to be the set of those additive functions $g : \mb{N} \ra \mb{C}$ such that
\begin{enumerate}[(a)]
\item $B_g(X) \ra \infty$, and 
\item $B_g(X)$ is dominated by the prime values $|g(p)|$, in the sense that 
$$
\limsup_{X \ra \infty} \frac{1}{B_g(X)^2} \sum_{\ss{p^k \leq X \\ k \geq 2}} \frac{|g(p^k)|^2}{p^k} = 0.
$$
\end{enumerate}
\end{def1}
We shall see below (see Lemma \ref{lem:redtoSA}a)) that $\mc{A}$ contains all completely additive and all strongly additive\footnote{By a \emph{strongly additive} function we mean an additive function $g$ such that $g(p^k) = g(p)$ for all primes $p$ and all $k \geq 1$.} functions $g$ with $B_g(X) \ra \infty$. Within $\mc{A}$ we define
$$
\mc{A}_s := \{g \in \mc{A} : \lim_{\e \ra 0^+} F_g(\e) = 0\}.
$$
Thus, among other examples, $\Omega(n), \omega(n) := \sum_{p|n} 1$ and, for any $c \in \mb{C}$, $c\log$ all belong to $\mc{A}_s$.  We show in general that whenever $g \in \mc{A}_s$, we may obtain an $\ell^2$ analogue of Theorem \ref{thm:MRL1}.
\begin{thm}\label{thm:MRL2}
Let $g: \mb{N} \ra \mb{C}$ be an additive function in $\mc{A}_s$. Let $10 \leq h \leq X/100$ be an integer with $h = h(X) \ra \infty$. Then
$$
\frac{2}{X}\sum_{X/2 < n \leq X} \left|\frac{1}{h} \sum_{n-h < m \leq n} g(m) - \frac{2}{X}\sum_{X/2 < m \leq X} g(m) \right|^2 = o(B_g(X)^2).
$$
\end{thm}
Our proof of Theorem \ref{thm:MRL2} relies on a variant of the Matom\"{a}ki-Radziwi\l\l \ theorem that applies to a large collection of \emph{divisor-bounded} multiplicative functions, proven in the recent paper \cite{ManMR}. See Theorem \ref{thm:MRDB} below for a statement relevant to the current circumstances.

\begin{rem}
The rate of decay in this result depends implicitly on the rate of decay of $F_g(\e)$ and on the size of the contribution to $B_g(X)$ from the prime power values of $g$. We have therefore chosen to state the theorem in this qualitative form for the sake of simplicity.
\end{rem}

It deserves mention that the application of the Matom\"{a}ki-Radziwi\l\l \ method, which will be used in this paper, to the study of specific additive functions is not entirely new. Goudout \cite{Gou}, \cite{Gou2} applied this technique to derive distributional information about $\omega(n)$ in typical short intervals; for example, he proved in \cite{Gou} that the Erd\H{o}s-Kac theorem holds in short intervals $(x-h,x]$ for almost all $x \in [X/2,X]$, as long as $h = h(X) \ra \infty$. 
The specific novelty of Theorems \ref{thm:MRL1} and \ref{thm:MRL2} lie in their generality, and it is this aspect which will be used in the applications to follow.

\subsection{Applications: gaps and rigidity problems for additive functions} \label{subsec:apps}
Given $c \in \mb{C}$, the arithmetic function $n \mapsto c \log n$ is completely additive. In contrast to a typical additive function $g$, whose values $g(n)$ depend on the prime factorization of $n$ which might vary wildly from one integer to the next, $c\log$ varies slowly and smoothly, with very small gaps 
$$
c\log(n+1)-c\log n = O(1/n) \text{ for all } n \in \mb{N}.
$$
In the seminal paper \cite{Erd}, Erd\H{o}s studied various characterization problems for real- and complex-valued additive functions relating to their local behaviour, and in so doing found several characterizations of the logarithm as an additive function. Among a number of results, he showed that if either
\begin{enumerate}[(a)]
\item $g(n+1) \geq g(n)$ for all $n \in \mb{N}$, or
\item $g(n+1)-g(n) = o(1)$ as $n \ra \infty$
\end{enumerate}
then there exists $c \in \mb{R}$ such that $g(n) = c\log n$ for all $n \geq 1$. \\
Moreover, Erd\H{o}s and later authors posited that these hypotheses could be relaxed. K\'{a}tai \cite{Kat} and independently Wirsing \cite{Wir2} weakened assumption (b), and proved the above result under the averaged assumption
$$
\lim_{X \ra \infty} \frac{1}{X} \sum_{n \leq X} |g(n+1)-g(n)| = 0.
$$
Hildebrand \cite{Hil2} showed the stronger conjecture of Erd\H{o}s that if $g(n_k+1)-g(n_k) \ra 0$ on a set $\{n_k\}_k$ of density 1 then $g = c \log$; this, of course, is an \emph{almost sure} version of (b). \\
In a different direction, Wirsing \cite{Wir} showed that for completely additive functions $g$, (b) may be weakened to $g(n+1)-g(n) = o(\log n)$ as $n \ra \infty$, and this is best possible. \\
A number of these results were strengthened and generalized by Elliott \cite[Ch. 11]{Ell}, in particular to handle functions $g$ with small gaps $|g(an+b)-g(An+B)|$, for independent linear forms $n \mapsto an+b$ and $n \mapsto An+B$. \\
Characterization problems of these kinds for both additive and multiplicative functions have continued to garner interest more recently. In \cite{klu}, Klurman proved a long-standing conjecture of K\'{a}tai, showing that if a unimodular multiplicative function $f: \mb{N} \ra S^1$ has gaps $|f(n+1)-f(n)|\ra 0$ on average then for some $t \in \mb{R}$ we have $f(n) = n^{it}$ for all $n$. In a later work, Klurman and the author \cite{klu_man_rig} proved a conjecture of Chudakov from the '50s characterizing completely multiplicative functions having uniformly bounded partial sums. See K\'{a}tai's survey paper \cite{KatSur} for numerous prior works in this direction for both additive and multiplicative functions. \\
While these multiplicative results have consequences for additive functions, they are typically limited by the fact that if $g$ is a real-valued additive function then the multiplicative function $e^{2\pi i g}$ is only sensitive to the values $g(n) \pmod{1}$. In particular, considerations about e.g., the monotone behaviour of $g$ cannot be directly addressed by appealing to corresponding results for multiplicative functions. \\
\subsubsection{Erd\H{o}s' Conjecture for Almost Everywhere Monotone Additive Functions}
One still open problem stated in \cite{Erd} concerns the \emph{almost sure} variant of problem (a) above.  For convenience, given an additive function $g: \mb{N} \ra \mb{R}$ we set $g(0) := 0$ and define the set of decrease of $g$:
$$
\mc{B} := \{n \in \mb{N} : g(n) < g(n-1)\}, \quad \quad \mc{B}(X) := \mc{B} \cap [1,X].
$$
\begin{conj}[Erd\H{o}s, 1946 \cite{Erd}] \label{conj:Erd}
Let $g: \mb{N} \ra \mb{R}$ be an additive function, such that
\begin{equation}\label{eq:sparseB}
|\mc{B}(X)| = o(X) \text{ as $X \ra \infty$.}
\end{equation}
Then there exists $c \in \mb{R}$ such that $g(n) = c\log n$ for all $n \in \mb{N}$.
\end{conj}
Thus, if $g$ is non-decreasing except on a set of integers of natural density 0 then $g$ is a constant times a logarithm. \\
Condition \eqref{eq:sparseB} is necessary, as for any $\e > 0$ one can construct a function $g$, not a constant multiple of $\log n$, which is monotone except on a set of upper density at most $\e$. Indeed,
picking a prime $p_0 > 1/\e$ and defining $g = g_{p_0}$ to be the completely additive function given by
$$
g_{p_0}(p) := \begin{cases} \log p &: \ p \neq p_0 \\ p_0 &: p = p_0,\end{cases}
$$
one finds that $g_{p_0}(n) = \log n$ if and only if $ n \notin \mc{B} = \{mp_0 + 1: m \in \mb{N}\}$, with $0 < d\mc{B} = 1/p_0 < \e$.\\
As a consequence of our results on short interval averages of additive functions, we will prove the following partial result towards Erd\H{o}s' conjecture.
\begin{cor}\label{cor:weakErd}
Let $g: \mb{N} \ra \mb{R}$ be a completely additive function that satisfies
\begin{equation}\label{eq:smallLarge}
\lim_{\e \ra 0^+} F_g(\e) = \lim_{\e \ra 0^+} \limsup_{X \ra \infty} \frac{1}{B_g(X)^2} \sum_{\ss{p \leq X \\ |g(p)| > \e^{-1}B_g(X)}} \frac{g(p)^2}{p} = 0.
\end{equation}
Assume furthermore that there is a $\delta > 0$ such that 
$$
|\mc{B}(X)| \ll X/(\log X)^{2+\delta}. 
$$
Then there is a constant $c \in \mb{R}$ such that $g(n) = c\log n$ for all $n \in \mb{N}$.
\end{cor}
The above corollary reflects the fact that the main difficulties involved in fully resolving Conjecture \ref{conj:Erd} are: i) the possible lack of sparseness of $\mc{B}$ beyond $|\mc{B}(X)| = o(X)$, and ii) the possibility of very large values $|g(p)|$. \\
%
%
%
%
%
More generally, we show that any function $g \in \mc{A}_s$ is \emph{close} to a constant multiple of a logarithm at prime powers.
\begin{thm}\label{thm:iterStep}
Let $g:\mb{N} \ra \mb{R}$ be an additive function belonging to $\mc{A}_s$, and suppose $|\mc{B}(X)| = o(X)$. 
Let $X \geq 10$ be large. Then there is $\lambda = \lambda(X)$ with $|\lambda(X)| \ll B_g(X)/\log X$ such that
$$
\sum_{p^k \leq X} \frac{|g(p^k)-\lambda(X)\log p^k|^2}{p^k} = o\left(\sum_{p^k \leq X} \frac{g(p^k)^2}{p^k}\right), \text{ as } X \ra \infty.
$$
Moreover, $\lambda$ is slowly-varying as a function of $X$ in the sense that for every fixed $0 < u \leq 1$,
$$
\lambda(X) = \lambda(X^u) + o\left(\frac{B_g(X)}{\log X}\right).
$$
\end{thm}
As we will show in Section \ref{subsec:pfWeakErd}, the conclusion of Theorem \ref{thm:iterStep} and the slow variation of $\lambda$ as a function of $X$ are sufficient to imply that $B_g(X) = (\log X)^{1+o(1)}$ as $X \ra \infty$. Gaining control over the growth of $B_g(X)$ is crucial in establishing Corollary \ref{cor:weakErd}.

Finally, using a result of Elliott \cite{EllFA}, we will prove the following \emph{approximate} version of Erd\H{o}s' conjecture under weaker conditions than in Corollary \ref{cor:weakErd}.
\begin{thm}\label{thm:almErd}
Let $g: \mb{N} \ra \mb{R}$ be an additive function, such that $|\mc{B}(X)| = o(X)$. Then there are parameters $\lambda = \lambda(X)$ and $\eta = \eta(X)$ such that for all but $o(X)$ integers $n \leq X$,
\begin{equation}\label{eq:almostErd}
g(n) = \lambda \log n - \eta + o(B_g(X)).
\end{equation}
The functions $\lambda,\eta$ are slowly-varying in the sense that for any $u \in (0,1)$ fixed,
$$
\lambda(X^u) = \lambda(X) + o\left(\frac{B_g(X)}{\log X}\right), \quad\quad \eta(X^u) = \eta(X) + o(B_g(X)).
$$
\end{thm}
\begin{rem}
Note that if we knew \eqref{eq:almostErd} held for all three of $n,m,nm \in [1,X]$ then we could deduce that 
$$
\lambda \log(nm) - 2\eta + o(B_g(X)) = g(n) + g(m) = g(nm) = \lambda \log(nm) - \eta + o(B_g(X)),
$$
and thus that $\eta = o(B_g(X))$. As such, \eqref{eq:almostErd} would be valid upon taking $\eta \equiv 0$. Unfortunately, we are not able to confirm this unconditionally. 
\end{rem}

\subsubsection{On Elliott's Property of Gaps}
Gap statistics provide an important example of local properties of a sequence. Obviously, an additive function $g$ whose values $g(n)$ are globally close to $g$'s mean value must have small gaps $|g(n+1)-g(n)|$. Conversely, it was observed by Elliott that the growth of the gaps between consecutive values of $g$ also control the typical discrepancy of $g(n)$ from its mean. 
\begin{thm1}[Elliott \cite{Ell}, Thm. 10.1]
There is an absolute constant $c > 0$ such that for any additve function $g: \mb{N} \ra \mb{C}$ one has
$$
\frac{1}{X} \sum_{n \leq X} |g(n) - A_g(X)|^2 \ll \sup_{X \leq y \leq X^c} \frac{1}{y} \sum_{n \leq y} |g(n)-g(n-1)|^2,
$$
where $A_g(X) := \sum_{p^k \leq X} g(p^k) p^{-k} (1-1/p)$ is the asymptotic mean value of $g$.
\end{thm1}
Hildebrand \cite{Hil3} showed that any $c > 4$ is admissible. 
Elliott's result shows that if $g$ has exceedingly small gaps on average, even at scales that \emph{grow polynomially in $X$}, then $g$ must globally be very close to its mean. \\
The drawback of this result is that it is in principle possible for the upper bound to be trivial even if the gaps $|g(n+1)-g(n)|$, $n \leq X$, are $o(B_g(X))$ on average, as long as the average savings over $n \leq X^c$ is not large enough to offset the difference in size between $B_g(X)$ and $B_g(X^c)$.\\
In Section \ref{sec:gaps}, we obtain two results that complement Elliott's. The first shows that for any additive function $g$, \emph{any} savings in the $\ell^1$-averaged moment of $|g(n)-g(n-1)|$ provides a savings over the trivial bound for the first centred moment. The second, which holds whenever $g \in \mc{A}_s$, gives the same type of information as the first but in an $\ell^2$ sense.
\begin{thm}\label{thm:Elltype}
Let $g: \mb{N} \ra \mb{C}$ be an additive function.  
\begin{enumerate}[(a)]
\item We have 
$$
\frac{1}{X} \sum_{n \leq X} |g(n)-g(n-1)| = o(B_g(X)) \text{ if, and only if, } \frac{1}{X} \sum_{n\leq X} |g(n)-A_g(X)| = o(B_g(X)).
$$
\item Assume furthermore that $g \in \mc{A}_s$. 
Then
$$
\frac{1}{X} \sum_{n \leq X} |g(n)-g(n-1)|^2 = o(B_g(X)^2) \text{ if, and only if, }
\frac{1}{X} \sum_{n \leq X} |g(n)-A_g(X)|^2 = o(B_g(X)^2).
$$
\end{enumerate}
\end{thm}
See Proposition \ref{prop:convGaps} below, where an explicit dependence between the rates of decay of the gap average and the first centred moment in Theorem \ref{thm:Elltype}(a) is given as a consequence of Theorem \ref{thm:MRL1}.

\begin{rem}
Even in the weak sense of Theorem \ref{thm:Elltype} and even when $g$ takes bounded values at primes, it can be seen that having small gaps on average is a very special property. As a simple example, $g = \omega$, for which $B_{\omega}(X) \sim \sqrt{\log\log X}$, satisfies
$$
\frac{1}{X}\sum_{n \leq X} |\omega(n)-\omega(n-1)| \gg \sqrt{\log\log X},
$$
since by a bivariate version of the Erd\H{o}s-Kac theorem (see e.g., \cite{ManEK}) one can find a positive proportion of integers $n \in [X/2,X]$ such that, simultaneously,
$$
\frac{\omega(n)-\log\log X}{\sqrt{\log\log X}} \geq 2, \quad\quad \frac{\omega(n-1) - \log\log X}{\sqrt{\log\log X}} \leq 1.
$$
In fact, Theorem \ref{thm:Elltype} turns out to imply (in a somewhat weak sense) that if $g$ has a small $\ell^2$ average gap then by a refinement of the Tur\'{a}n-Kubilius inequality due to Ruzsa \cite{Ruz} (see Lemma \ref{lem:Ruzsa} below), $g$ must behave like a constant times a logarithm on average over prime powers.
\end{rem}


\section{Auxiliary Lemmas}
In this section we record several results that will be used repeatedly in the sequel.
\begin{lem}\label{lem:mean}
Let $g: \mb{N} \ra \mb{C}$ be additive. Then for any $Y \geq 3$,
$$
\frac{1}{Y}\sum_{n \leq Y} g(n) = A_g(Y) + O\left(\frac{B_g(Y)}{\sqrt{\log Y}}\right).
$$
\end{lem}
\begin{proof}
We have
\begin{align*}
&\frac{1}{Y}\sum_{n \leq Y} g(n) - A_g(Y) = \sum_{p^k \leq Y} g(p^k)\left(Y^{-1}\left(\llf \frac{Y}{p^k} \rrf - \llf \frac{Y}{p^{k+1}}\rrf\right) - \frac{1}{p^k}(1-1/p)\right) \\
&\ll \sum_{p^k \leq Y} |g(p^k)| \leq Y^{1/2} \sum_{p^k \leq Y} |g(p^k)|p^{-k/2} \ll B_g(Y) (Y\pi(Y))^{1/2} \ll \frac{B_g(Y)}{\sqrt{\log Y}},
\end{align*}
using the Cauchy-Schwarz inequality and Chebyshev's estimate $\pi(Y) \ll Y/\log Y$ in the last two steps.
\end{proof}

\begin{lem}[Tur\'{a}n-Kubilius Inequality] \label{lem:TK}
Let $X \geq 3$. Uniformly over all additive functions $g: \mb{N} \ra \mb{C}$,
$$
\frac{1}{X}\sum_{n\leq X} |g(n)-A_g(X)|^2 \ll B_g(X)^2.
$$
\end{lem}
\begin{proof}
This is e.g., \cite[Lem. 1.5]{Ell} (taking $\sg = 0$).
\end{proof}

The following estimate due to Ruzsa, which sharpens the Tur\'{a}n-Kubilius inequality, gives an order of magnitude estimate for the second centred moment of a general additive function.
\begin{lem}[Ruzsa \cite{Ruz}]\label{lem:Ruzsa}
Let $g: \mb{N} \ra \mb{C}$ be an additive function. Then
$$
\frac{1}{X} \sum_{n \leq X} |g(n)-A_g(X)|^2 \asymp \min_{\lambda \in \mb{R}} (B_{g_{\lambda}}(X)^2 + \lambda^2) \asymp B_{g_{\lambda_0}}(X)^2 + \lambda_0^2,
$$
where $\lambda_0 = \lambda_0(X)$ is given by
$$
\lambda_0(X) = \frac{2}{(\log X)^2} \sum_{p \leq X} \frac{g(p)\log p}{p}.
$$
\end{lem}

\begin{lem} \label{lem:Ag}
Let $g: \mb{N} \ra \mb{C}$ be additive, and let $1 \leq y/2 \leq z \leq y$. Then
$$
A_g(z) = A_g(y) + O\left(\frac{B_g(y)}{\sqrt{\log y}}\right).
$$
\end{lem}
\begin{proof}
We have
$$
|A_g(y)-A_g(z)| \leq \sum_{z < p^k \leq y} \frac{|g(p^k)|}{p^k} \leq B_g(y)\left(\sum_{y/2 < p^k \leq y} \frac{1}{p^k}\right)^{1/2} \ll \frac{B_g(y)}{\sqrt{\log y}},
$$
where the last estimate follows from Mertens' theorem.
\end{proof}

\begin{lem}\label{lem:ptwiseBdg}
Let $X\geq 3$ and let $n \in (X/2,X]$. Then $\frac{|g(n)|}{n} \ll \frac{B_g(X)\log X}{\sqrt{X}}$.
\end{lem}
\begin{proof}
Observe that whenever $p^k \leq X$ we have $g(p^k)/p^{k/2} \leq B_g(X)$. It follows from the triangle inequality and $\omega(n) \leq \log n$ that
$$
\frac{|g(n)|}{n} \leq \frac{\omega(n)}{n} \max_{p^k||n} |g(p^k)| \leq \frac{\omega(n)}{\sqrt{n}} \max_{p^k || n} \frac{|g(p^k)|}{p^{k/2}} \leq \frac{2\log X}{\sqrt{X}} B_g(X),
$$
and the claim follows.
\end{proof}

Working within the collection $\mc{A}$, the following properties will be useful.
\begin{lem}\label{lem:redtoSA}
a) Let $g: \mb{N} \ra \mb{C}$ be an additive function satisfying $B_g(X) \ra \infty$. If $g$ is either completely or strongly additive then $g \in \mc{A}$. \\
b) Let $g \in \mc{A}$. Then there is a strongly additive function $g^{\ast}$ such that $B_{g-g^{\ast}}(X) = o(B_g(X))$ as $X \ra \infty$.
\end{lem}
\begin{proof}
a) Let $g$ be either strongly or completely additive. We put $\theta_g := 1$ if $g$ is completely additive, and $\theta_g := 0$ otherwise. Then
$$
\sum_{\ss{p^k \leq X \\ k \geq 2}} \frac{|g(p^k)|^2}{p^k} \leq \sum_{p \leq X} \frac{|g(p)|^2}{p} \sum_{k \geq 2} \frac{k^{2\theta_g}}{p^{k-1}} \ll \sum_{p \leq X} \frac{|g(p)|^2}{p^2}.
$$
Since $B_g(X) \ra \infty$, choosing $M = M(X)$ tending to infinity arbitrarily slowly we see that
$$
\sum_{p \leq X} \frac{|g(p)|^2}{p^2} \leq \sum_{p \leq M} \frac{|g(p)|^2}{p} + \frac{1}{M} \sum_{M < p \leq X} \frac{|g(p)|^2}{p} \leq B_g(M)^2 + \frac{B_g(X)^2}{M} = o(B_g(X)^2).
$$
It follows that $g \in \mc{A}$, as required. \\
b) We define $g^{\ast}(p^k) := g(p)$ for all primes $p$ and $k \geq 1$, so that $g^{\ast}$ is strongly additive; moreover, if $(g-g^{\ast})(p^k) \neq 0$ then $k \geq 2$, for any $p$. By assumption and part a), $g,g^{\ast} \in \mc{A}$. Moreover, the argument in a) implies that
$$
\sum_{\ss{p^k \leq X \\ k \geq 2}} \frac{|g^{\ast}(p^k)|^2}{p^k} \ll \sum_{p \leq X} \frac{|g(p)|^2}{p^2} = o(B_g(X)^2).
$$
It follows from the Cauchy-Schwarz inequality that
$$
B_{g-g^{\ast}}(X)^2 \ll \sum_{\ss{p^k \leq X \\ k \geq 2}} \frac{|g(p^k)|^2}{p^k} + \sum_{\ss{p^k \leq X \\ k \geq 2}} \frac{|g^{\ast}(p^k)|^2}{p^k} = o(B_g(X)^2),
$$
%
%
%
as required.
\end{proof}

Finally, we record the characterization result of K\'{a}tai and Wirsing, mentioned in the introduction.
\begin{thm}[K\'{a}tai \cite{Kat}, Wirsing \cite{Wir}] \label{thm:KatWir}
Let $g: \mb{N} \ra \mb{C}$ be an additive function such that
$$
\frac{1}{X}\sum_{n \leq X} |g(n)-g(n-1)| = o(1),
$$
as $X \ra \infty$. Then there is $c \in \mb{C}$ such that $g(n) = c\log n$ for all $n \in \mb{N}$.
\end{thm}

\section{The Matom\"{a}ki-Radziwi\l\l\ \  Theorem for Additive Functions: $\ell^1$ Variant} \label{sec:L1MR}
In this section, we prove Theorem \ref{thm:MRL1}.
We begin with the following simple observation, amounting to the fact that the mean value of an additive function changes little when passing from a long interval of length $\asymp X$ to a medium-sized one of length $X/(\log X)^{\delta}$, for $\delta > 0$ not too large. 
\begin{lem} \label{lem:longH}
Let $g: \mb{N} \ra \mb{C}$ be additive and let $X$ be large. Let $X/2 < x \leq X$, and let $X/(\log X)^{1/3} \leq h \leq X/3$. Then
$$
\frac{2}{X} \sum_{X/2 < n \leq X} g(n) = \frac{1}{h} \sum_{x-h \leq n \leq x} g(n) + O\left(\frac{B_g(X)}{(\log X)^{1/6}}\right).
$$
\end{lem}
\begin{proof}
Applying Lemma \ref{lem:mean} with $Y = X/2,X,x-h$ and $x$, we obtain
\begin{align*}
&\frac{2}{X}\sum_{X/2 < n \leq X} g(n) = 2A_g(X) - A_g(X/2) + O\left(\frac{B_g(X)}{\sqrt{\log X}}\right) \\
&\frac{1}{h}\sum_{x-h < n \leq x} g(n) = \frac{x}{h}A_g(x) - \left(\frac{x}{h}-1\right)A_g(x) + O\left(\frac{XB_g(X)}{h\sqrt{\log X}}\right).
\end{align*}
Since $X/(\log X)^{1/3} \leq h \leq X/2$, the error term in the second line is $\ll \frac{B_g(X)}{(\log X)^{1/6}}$. By Lemma \ref{lem:Ag},
$$
2A_g(X)-A_g(X/2) = A_g(X) + O(|A_g(X)-A_g(X/2)|) = A_g(X) + O\left(\frac{B_g(X)}{\sqrt{\log X}}\right)
$$
for the main term in the first equation, and also
$$
\frac{x}{h}|A_g(x)-A_g(x-h)| \ll (\log x)^{1/3} \cdot \frac{B_g(x)}{\sqrt{\log x}} \ll \frac{B_g(X)}{(\log X)^{1/6}},
$$
so that by a second application of Lemma \ref{lem:Ag},
$$
\frac{x}{h}A_g(x) - \left(\frac{x}{h}-1\right)A_g(x-h) = A_g(x) + O\left(\frac{x}{h}|A_g(x)-A_g(x-h)|\right) = A_g(X) + O\left(\frac{B_g(X)}{(\log X)^{1/6}}\right).
$$
Combining these estimates, we may conclude that
$$
\left|\frac{2}{X}\sum_{X/2 < n \leq X} g(n)-\frac{1}{h} \sum_{x-h < n \leq x} g(n)\right| 
\ll \frac{B_g(X)}{(\log X)^{1/6}},
$$
as claimed.
\end{proof}
In light of the above lemma, it suffices to prove the following: if $h' = X/(\log X)^{1/3}$ and $10 \leq h \leq h'$ then
$$
\frac{2}{X}\sum_{X/2<m \leq X}\left|\frac{1}{h}\sum_{m-h < n \leq m} g(n)-\frac{1}{h'} \sum_{m-h' < n \leq m} g(n)\right|  \ll B_g(X) \left(\sqrt{\frac{\log\log h}{\log h}} + (\log X)^{-1/800}\right).
$$
Splitting $g = \text{Re}(g) + i \text{Im}(g)$, and noting that both $\text{Re}(g)$ and $\text{Im}(g)$ are real-valued additive functions, we may assume that $g$ is itself real-valued, after which the general case will follow by the triangle inequality. \\
Let $10 \leq h \leq X/3$, with $X$ large. Following \cite{MR}, fix $\eta \in (0,1/12)$, parameters $Q_1 = h$, $P_1 = (\log h)^{40/\eta}$, and define further parameters $P_j,Q_j$ by 
$$
P_j := \exp\left(j^{4j} (\log Q_1)^{j-1}\log P_1\right), \quad Q_j := \exp\left(j^{4j+2} (\log Q_1)^j\right),
$$
for all $j \leq J$, where $J$ is chosen maximally subject to $Q_J \leq \exp(\sqrt{\log X})$. We then define
$$
\mc{S} = \mc{S}_{X,P_1,Q_1} := \{n \leq X : \omega_{[P_j,Q_j]}(n) \geq 1 \text{ for all $1 \leq j \leq J$}\},
$$
where for any set $S \subset \mb{N}$ we write $\omega_S(n) := \sum_{p|n} 1_S(n)$. 
\begin{lem} \label{lem:passtoMR}
Let $g: \mb{N} \ra \mb{R}$ be an additive function. Let $10 \leq h \leq h'$, where $h' := \frac{X}{(\log X)^{1/3}}$ and $h \in \mb{Z}$, and $(\log X)^{-1/6} < t < 1$. Then
\begin{align*}
&\frac{2}{X}\sum_{X/2< m \leq X} \left|\frac{1}{h'} \sum_{\substack{m-h' < n \leq m}} g(n) - \frac{1}{h} \sum_{\substack{m-h < n \leq m}} g(n)\right|  \\
&\ll \frac{B_g(X)}{t}\cdot \frac{1}{X} \int_{X/2}^X \left|\frac{1}{h'} \sum_{\substack{x-h' < n \leq x \\ n \in \mc{S}}} e(t\tilde{g}(n;X)) - \frac{1}{h} \sum_{\substack{x-h < n \leq x\\ n \in \mc{S}}} e(t\tilde{g}(n;X)) \right|  dx + B_g(X)\left(t + \frac{\log\log h}{t\log h}\right),
\end{align*}
where $\tilde{g}(n;X) := B_g(X)^{-1}(g(n) - A_g(X))$ for all $n \in \mb{N}$.
\end{lem}
\begin{proof}
In view of Lemma \ref{lem:ptwiseBdg}, at the cost of an error term of size $|g(\lceil x-h'\rceil)|/h' \ll B_g(X)X^{-1/4}$, we may assume that both $h,h' \in \mb{Z}$ (else replace $h'$ by $\llf h'\rrf$). Given $u \in [0,1]$, $x \in [X/2,X] \cap \mb{Z}$ and an integer $1 \leq H \leq h'$, define
$$
S_{H}(u;x) := \frac{1}{H}\sum_{\substack{x-H < n \leq x }} e(u\tilde{g}(n;X)),
$$
which is clearly an analytic function of $u$. Fix $x \in [X/2,X] \cap \mb{Z}$, and observe that $S_{h'}(0;x) = 1 = S_h(0;x)$. 
By Taylor expansion,
$$
S_{h'}(t;x)-S_h(t;x) = t(S_{h'}'(0;x)-S_h'(0;x)) + \frac{1}{2}\int_0^t (S_{h'}''(u;x)-S_h''(u;x)) u du,
$$
wherein we have
\begin{align*}
S_{h'}'(0;x) - S_h'(0;x) &= \frac{1}{B_g(X)} \left(\frac{1}{h'} \sum_{\substack{x-h' < n \leq x }} (g(n)-A_g(X)) - \frac{1}{h} \sum_{\substack{x-h < n \leq x }}(g(n)-A_g(X))\right) \\
&= \frac{1}{B_g(X)} \left(\frac{1}{h'} \sum_{\substack{x-h' < n \leq x }} g(n) - \frac{1}{h} \sum_{\substack{x-h < n \leq x }} g(n)\right).
\end{align*}
Solving for $S_{h'}'(0;x)-S_h'(0;x)$, taking absolute values and then averaging over $x \in [X/2,X] \cap \mb{Z}$, we get 
\begin{align*}
&\frac{2}{X}\sum_{X/2<m \leq X} \left|\frac{1}{h'}\sum_{\substack{m-h' < n \leq m }} g(n) - \frac{1}{h} \sum_{\substack{m-h < n \leq m }} g(n)\right|  \\
&\ll B_g(X) t^{-1} \cdot \frac{1}{X} \sum_{X/2<m\leq X} \left|\frac{1}{h'} \sum_{\substack{m-h' < n \leq m }} e(t\tilde{g}(n;X)) - \frac{1}{h} \sum_{\substack{m-h < n \leq m }} e(t\tilde{g}(n;X))\right| \\
&+ B_g(X) t^{-1}\cdot \frac{t^2 }{X} \sum_{X/2 < m \leq X}\max_{0 \leq u \leq t} \left|\frac{1}{h'} \sum_{\substack{m-h' < n \leq m }} \tilde{g}(n;X)^2 e(u\tilde{g}(n;X)) - \frac{1}{h} \sum_{\substack{m-h < n \leq m }} \tilde{g}(n;X)^2 e(u\tilde{g}(n;X))\right|.
\end{align*}
Since $g$ is real-valued by assumption, $|e(u\tilde{g}(n;X))| = 1$ for all $n$. Thus, applying the triangle inequality and Lemma \ref{lem:TK}, we may bound the last expression above by
\begin{align*}
&\ll t B_g(X) \frac{1}{X} \sum_{X/2 < m \leq X} \sum_{H \in \{h,h'\}} \frac{1}{H} \sum_{m-H < n \leq m} \left(\frac{g(n)-A_g(X)}{B_g(X)}\right)^2 \\
&\ll tB_g(X) \cdot \frac{1}{X} \sum_{X/2 < n \leq X} \left(\frac{g(n)-A_g(X)}{B_g(X)}\right)^2 \cdot \sum_{H \in \{h,h'\}} \frac{1}{H}\sum_{X/2 < m \leq X} 1_{[n,n+H]}(m) \\
&\ll tB_g(X).
\end{align*}
We now split 
$$
S_H(t;x) = \frac{1}{H}\sum_{\substack{x-H < n \leq x \\ n\in \mc{S}}} e(u\tilde{g}(n;X)) + \frac{1}{H}\sum_{\substack{x-H < n \leq x \\ n\notin \mc{S}}} e(u\tilde{g}(n;X)) =: S_H^{(\mc{S})}(t;x) + S_H^{(\mc{S}^c)}(t;x),
$$
with $H\in \{h,h'\}$. By the triangle inequality, we have
$$
\frac{2}{X}\sum_{X/2 < x \leq X} |S_H^{(\mc{S}^c)}(t;x)| \leq \frac{1}{HX}\sum_{X/2 < x \leq X} |\mc{S}^c \cap (x-H,x]| \leq \frac{1}{X}\sum_{X/3 < n \leq X}1_{\mc{S}^c}(n).
$$
Since $P_J \leq \exp(\sqrt{\log X})$, the union bound and the fundamental lemma of the sieve yield
\begin{align}
|\mc{S}^c \cap [X/3,X]| &\ll X\sum_{1 \leq j \leq J}\prod_{P_j \leq p \leq Q_j} \left(1-\frac{1}{p}\right) \ll X\sum_{1 \leq j \leq J}\frac{\log P_j}{\log Q_j} = X\frac{\log P_1}{\log Q_1} \sum_{1\leq j \leq J} \frac{1}{j^2} \ll X\frac{\log\log h}{\log h}. \label{eq:notinS}
\end{align}
We thus find by the triangle inequality that
$$
\frac{2}{X}\sum_{X/2 < n \leq X} \left|S_h(t;n) - S_{h'}(t;n)\right| \ll \frac{2}{X}\sum_{X/2 < n \leq X} \left|S_h^{(\mc{S})}(t;n) - S_{h'}^{(\mc{S})}(t;n)\right| + \frac{\log\log h}{\log h}.
$$
Finally, if $n \in [X/2,X] \cap \mb{Z}$ and $x \in [n,n+1)$ then $S_H^{(\mc{S})}(t;x) = S_H^{(\mc{S})}(t; n) + O(1/H)$, and thus
$$
\frac{2}{X}\sum_{X/2 < n \leq X} |S_h^{(\mc{S})}(t;x) - S_{h'}^{(\mc{S})}(t;x)| \leq \frac{2}{X}\int_{X/2}^X |S_h^{(\mc{S})}(t;x) - S_{h'}^{(\mc{S})}(t;x)| dx + O(1/h).
$$ 
Combined with the preceding estimates, we obtain
\begin{align*}
&\frac{2}{X}\sum_{X/2 < m \leq X} \left|\frac{1}{h'}\sum_{\substack{m-h' < n \leq m }} g(n) - \frac{1}{h} \sum_{\substack{m-h < n \leq m }} g(n)\right| dx \\
&\ll t^{-1} B_g(X)\left(\frac{2}{X}\int_{X/2}^X \left|S_h^{(\mc{S})}(t;x) - S_{h'}^{(\mc{S})}(t;x)\right| dx + \frac{\log\log h}{\log h}\right) + tB_g(X),
\end{align*}
which implies the claim.
\end{proof}
Let $\mb{U} := \{z \in \mb{C} : |z| \leq 1\}$. In what follows, given multiplicative functions $f,g: \mb{N} \ra \mb{U}$ and parameters $1 \leq T \leq X$, we introduce the pretentious distance of Granville and Soundararajan:
\begin{align*}
\mb{D}(f,g; X)^2 &:= \sum_{p \leq X} \frac{1-\text{Re}(f(p)\bar{g}(p))}{p},\\
M_f(X;T) &:= \min_{|\lambda| \leq T} \mb{D}(f,n^{i\lambda}; X)^2.
\end{align*}
For multiplicative functions $f,g,h$ taking values in $\mb{U}$, $\mb{D}$ satisfies the triangle inequality (see e.g., \cite[Lem 3.1]{GSPret})
\begin{equation}\label{eq:trieqD}
\mb{D}(f,h;X) \leq \mb{D}(f,g;X) + \mb{D}(g,h;X).
\end{equation}
Define the multiplicative function
$$
G_{t,X}(n) := e(tg(n)/B_g(X)) = e(t \tilde{g}(n;X)) e(tA_g(X)/B_g(X)).
$$ 
For each $t \in [0,1]$, select $\lambda_{t,X} \in [-X,X]$ such that $M_{G_{t,X}}(X;X) = \mb{D}(G_{t,X},n^{i\lambda_{t,X}}; X)^2$. 
\begin{lem} \label{lem:controlMT}
Let 
$0 < t \leq 1/100$ be sufficiently small and let $(\log X)^{-1/100} < \e < 1/3$. Then either: 
\begin{enumerate}[(i)]
\item $M_{G_{t,X}}(X;X) \geq 4\log(1/\e)$, or else
\item $|\lambda_{t,X}| = O(1)$.
\end{enumerate}
\end{lem}
\begin{proof}
Assume i) fails. Then by assumption, $\mb{D}(G_{t,X}, n^{i\lambda_{t,X}};X)^2 \leq 4 \log(1/\e)$.
We claim that there is also $\tilde{\lambda}_{t,X} = O(1)$ such that 
\begin{align}\label{eq:pret}
\mb{D}(G_{t,X}, n^{i\tilde{\lambda}_{t,X}};X) \ll 1.
\end{align} 
To see that this is sufficient to prove (ii), we apply \eqref{eq:trieqD} to obtain
$$
\mb{D}(n^{i\lambda_{t,X}},n^{i\tilde{\lambda}_{t,X}};X) \leq 1+2\sqrt{\log(1/\e)} \leq 3\sqrt{\log(1/\e)}.
$$
Now, if $|\lambda_{t,X} - \tilde{\lambda}_{t,X}| \geq 100$ then as $|\lambda_{t,X}|,|\tilde{\lambda}_{t,X}| \leq X$ the Vinogradov-Korobov zero-free region for $\zeta$ (see e.g., \cite[(1.12)]{MRT}) gives
$$
\mb{D}(n^{i\lambda_{t,X}}, n^{i\tilde{\lambda}_{t,X}}; X)^2 = \log\log X - \log |\zeta(1+1/\log X + i(\lambda_{t,X} - \tilde{\lambda}_{t,X}))| + O(1) \geq 0.33 \log\log X,
$$
which is a contradiction given $\e^{-1} \leq (\log X)^{1/100}$. It follows that 
$$
|\lambda_{t,X}| \leq |\tilde{\lambda}_{t,X}| + 100 = O(1),
$$
as required. \\
It thus remains to prove that \eqref{eq:pret} holds. By Lemma \ref{lem:TK}, we obtain
$$
|\{n \leq X : |\tilde{g}(n;X)| > t^{-1/2} \}| \leq t\sum_{n \leq X} \tilde{g}(n;X)^2 \ll tX.
$$
It follows from Taylor expansion that
$$
\sum_{n \leq X} e(t\tilde{g}(n;X)) = \sum_{n \leq X}\left(1+O(\sqrt{t})\right) + O(tX) = (1+O(\sqrt{t}))X.
$$
On the other hand, by Hal\'{a}sz' theorem in the form of Granville and Soundararajan \cite[Thm. 1]{GSDec},
$$
\left|\sum_{n \leq X} e(t\tilde{g}(n;X))\right| = \left|\sum_{n \leq X} G_{t,X}(n)\right| \ll  M_{G_{t,X}}(X; U)e^{-M_{G_{t,X}}(X; U)} X  + \frac{X}{U},
$$
where $1 \leq U \leq \log X$ is a parameter of our choice.
If $U$ is a suitably large absolute constant and $t$ is sufficiently small in an absolute sense, we obtain $M_{G_{t,X}}(X; U) \ll 1$, and therefore there is a $\tilde{\lambda}_{t,X} \in [-U,U]$ (thus of size $O(1)$) such that 
$$
\mb{D}(G_{t,X},n^{i\tilde{\lambda}_{t,X}};X) \ll 1,
$$
as claimed.
\end{proof}

\begin{lem} \label{lem:MRapp}
Let $g: \mb{N} \ra \mb{R}$ be an additive function.
Let $X \geq 3$ be large, $(\log X)^{-1/6} < t \leq 1/100$ be small and let $10 \leq h_1 \leq h_2$ where $h_2 = X/(\log X)^{1/3}$. Then
\begin{align} \label{eq:Ssum}
&\frac{2}{X}\int_{X/2}^{X} \left|\frac{1}{h_1}\sum_{\substack{x-h_1 < n \leq x \\ n \in \mc{S}}} G_{t,X}(n) - \frac{1}{h_2}\sum_{\substack{x-h_2 < n \leq x \\ n \in \mc{S}}} G_{t,X}(n) \right| dx \ll  B_g(X)\left( \frac{\log\log h_1}{\log h_1} + (\log X)^{-1/400}\right).
\end{align}
\end{lem}
\begin{proof}
Set $\e = (\log X)^{-1/100}$. 
If $M_{G_{t,X}}(X;X) \geq 4\log(1/\e)$ then by the triangle inequality, Cauchy-Schwarz and \cite[Theorem A.2]{MRT}, the LHS of \eqref{eq:Ssum} is 
\begin{align*}
&\ll \sum_{j = 1,2} \left(\frac{1}{X}\int_{X/2}^{X} \left|\frac{1}{h_j}\sum_{\substack{x-h_j < n \leq x \\ n \in \mc{S}}} G_{t,X}(n) \right|^2 dx\right)^{1/2} \\
&\ll \exp\left(-\frac{1}{2}M_{G_{t,X}}(X;X)\right)M_{G_{t,X}}(X;X)^{1/2} + \frac{(\log h_1)^{1/6}}{P_1^{1/12-\eta/2}} + (\log X)^{-1/200} \\
&\ll \e^2 (\log(1/\e))^{1/2} + (\log h_1)^{-1} +(\log X)^{-1/200} \\
&\ll (\log h_1)^{-1} +(\log X)^{-1/200}.
\end{align*}
Next, assume that $M_{G_{t,X}}(X;X) < 4\log(1/\e)$. By Lemma \ref{lem:controlMT} we have $\lambda_{t,X} = O(1)$, so that 
\begin{align*}
&\frac{1}{h}\int_{x-h}^{x} u^{i\lambda_{t,X}} du = x^{i\lambda_{t,X}} \left(\frac{1-(1-h/x)^{1+i\lambda_{t,X}}}{(1+i\lambda_{t,X})h/x}\right) = x^{i\lambda_{t,X}} \left(1+O\left(|\lambda_{t,X}|\frac{h}{X}\right)\right) = x^{i\lambda_{t,X}}\left(1+O\left(\frac{h}{X}\right)\right),
\end{align*}
and thus for each $x \in [X/2,X]$ and $j = 1,2$,
\begin{equation} \label{eq:rightMT}
x^{i\lambda_{t.X}} \frac{2}{X} \sum_{\ss{ X/2 < n \leq X \\ n \in \mc{S}}} G_{t,X}(n) = \frac{1}{h_j} \int_{x-h_j}^x u^{i\lambda_{t,X}} du \cdot \frac{2}{X}\sum_{\ss{X/2 < n \leq X \\ n \in \mc{S}}} G_{t,X}(n) + O(h_j/X).
\end{equation}
Reinstating the $n \notin \mc{S}$, we also note the bound
\begin{align*}
&\frac{2}{X}\int_{X/2}^X \left|\frac{1}{h_j} \sum_{\ss{x-h_j < n \leq x \\ n \in \mc{S}}} G_{t,X}(n) - x^{i\lambda_{t.X}} \frac{2}{X} \sum_{\ss{ X/2 < n \leq X \\ n \in \mc{S}}} G_{t,X}(n)\right| dx \\
&\ll \frac{2}{X}\int_{X/2}^X \left|\frac{1}{h_j} \sum_{\ss{x-h_j < n \leq x }} G_{t,X}(n) - x^{i\lambda_{t.X}} \frac{2}{X} \sum_{\ss{ X/2 < n \leq X }} G_{t,X}(n)\right| dx + \frac{\log\log h_1}{\log h_1},
\end{align*}
which follows by using the arguments surrounding \eqref{eq:notinS}. \\
Adding and subtracting the expression on the LHS of \eqref{eq:rightMT} inside the absolute values bars in \eqref{eq:Ssum}, we obtain the upper bound
\[
\ll \mc{T}_1 + \mc{T}_2 + h_2/X
\]
where for $j = 1,2$ we set
$$
\mc{T}_j := \frac{1}{X}\int_{X/2}^{X} \left|\left(\frac{1}{h_j}\int_{x-h_j}^x u^{i\lambda_{t,X}} du\right) \frac{1}{X}\sum_{\substack{X/2 < n \leq X \\ n \in \mc{S}}} G_{t,X}(n)n^{-i\lambda_{t,X}} - \frac{1}{h_j}\sum_{\substack{x-h_j < n \leq x \\ n \in \mc{S}}} G_{t,X}(n) \right|.
$$
If $j = 2$ then we also have
$$
\mc{T}_j \ll \frac{2}{X}\int_{X/2}^X \left|\frac{1}{h_2} \sum_{\ss{x-h_2 < n \leq x }} G_{t,X}(n) - \left(\frac{1}{h_2}\int_{x-h_2}^x u^{i\lambda_{t,X}} du\right)\frac{2}{X} \sum_{\ss{ X/2 < n \leq X }} G_{t,X}(n)\right| dx + \frac{\log\log h_1}{\log h_1},
$$
and so by Cauchy-Schwarz and \cite[Theorem 1.6]{MRAP} (taking $Q = 1$ and $\e = (\log X)^{-1/200}$ there), we have
\begin{align*}
\mc{T}_2 &\ll \left(\frac{2}{X}\int_{X/2}^{X} \left|\left(\frac{1}{h_2}\int_{x-h_2}^x u^{i\lambda_{t,X}} du\right) \frac{1}{X}\sum_{\substack{X/2 < n \leq X }} G_{t,X}(n)n^{-i\lambda_{t,X}} - \frac{1}{h_j}\sum_{\substack{x-h_j < n \leq x }} G_{t,X}(n) \right|^2 dx\right)^{\frac 12} + \frac{\log\log h_1}{\log h_1} \\
&\ll (\log X)^{-1/400} + \frac{\log\log h_1}{\log h_1}.
\end{align*}
When $h_1 > \sqrt{X}$ we obtain the same bound $\mc{T}_1 \ll (\log X)^{-1/400} + \frac{\log\log h_1}{\log h_1}$ as well.\\
Thus, assume that $10 \leq h_1 \leq \sqrt{X}$. Combining Cauchy-Schwarz with \cite[Theorem 9.2(ii)]{MRII} (taking $\delta = (\log h_1)^{1/3}P_1^{-1/6+\eta}$, $\nu_1 = 1/20$ and $\nu_2 = 1/12$, there), we then get
$$
\mc{T}_1 \ll \left(\frac{(\log h_1)^{4/3}}{P_1^{1/12-\eta}} + (\log X)^{-1/200}\right)^{1/2} \ll (\log h_1)^{-1} + (\log X)^{-1/400}.
$$
%
%
%
%
Combining these estimates, we obtain that the LHS of \eqref{eq:Ssum} is
$$
\ll (\log h_1)^{-1} + (\log X)^{-1/400}+\frac{\log\log h_1}{\log h_1} + \frac{h_2}{X} \ll (\log X)^{-1/400} + \frac{\log\log h_1}{\log h_1},
$$
as claimed.
\end{proof}

\begin{proof}[Proof of Theorem \ref{thm:MRL1}]
Set $h_1 := h$ and $h_2 := X/(\log X)^{1/3}$. As mentioned, we may assume that $g$ is real-valued. By Lemma \ref{lem:longH}, we have
\begin{align}\label{eq:redtoMed}
&\frac{2}{X}\sum_{X/2 < m \leq X} \left|\frac{1}{h_1}\sum_{m-h_1 < n \leq m} g(n) - \frac{2}{X}\sum_{X/2  <n \leq X} g(n) \right| \nonumber \\
&\ll \frac{2}{X}\sum_{X/2 < m \leq X} \left| \frac{1}{h_1}\sum_{m-h_1 < n \leq m} g(n) - \frac{1}{h_2}\sum_{m-h_2 < n \leq m} g(n)\right| + \frac{B_g(X)}{(\log X)^{1/6}}. 
\end{align}
Observe next that for any $x \in [X/2,X]$ and $t \in (0,1)$,
\begin{align*}
S_{h_1}^{(\mc{S})}(t;x) - S_{h_2}^{(\mc{S})}(t;x)  &= \frac{1}{h_1} \sum_{\ss{x-h_1 < n \leq x \\ n \in \mc{S}}} e(t\tilde{g}(n;X)) - \frac{1}{h_2} \sum_{\ss{x-h_2 < n \leq x \\ n \in \mc{S}}} e(t\tilde{g}(n;X)) \\
&= e(-tA_g(X)/B_g(X)) \cdot \left(\frac{1}{h_1} \sum_{\ss{x-h_1 < n \leq x \\ n \in \mc{S}}} G_{t,X}(n) - \frac{1}{h_2} \sum_{\ss{x-h_2 < n \leq x \\ n \in \mc{S}}} G_{t,X}(n)\right). 
\end{align*}
Taking $t := \max\{\sqrt{\frac{\log\log h_1}{\log h_1}}, (\log X)^{-1/800}\}$, the proof of Theorem \ref{thm:MRL1} thus follows on combining Lemmas \ref{lem:MRapp} and \ref{lem:passtoMR} with \eqref{eq:redtoMed}.
\end{proof}

\section{A Conditional $\ell^2$ Matom\"{a}ki-Radziwi\l\l \ Theorem} \label{sec:L2MRAdd}
\noindent In this section, we will prove Theorem \ref{thm:MRL2}.
Let $g \in \mc{A}_s$, so that $B_g(X) \ra \infty$, and the conditions
\begin{align}
\lim_{\delta \ra 0^+} F_g(\delta) = \lim_{\delta \ra 0^+} \limsup_{X \ra \infty} \frac{1}{B_g(X)^2} \sum_{\ss{p \leq X \\ |g(p)| > \delta^{-1} B_g(X)}} \frac{|g(p)|^2}{p} &= 0 \label{eq:LindCond} \\
\limsup_{X \ra \infty} \frac{1}{B_g(X)^2} \sum_{\ss{p^k \leq X \\ k \geq 2}} \frac{|g(p^k)|^2}{p^k} &= 0.
\end{align}
both hold. We seek to show that
\begin{equation}\label{eq:L2mse}
\Delta_g(X,h) := \frac{2}{X}\sum_{X/2 < n \leq X} \left|\frac{1}{h} \sum_{n-h < m \leq n} g(m) - \frac{2}{X}\sum_{X/2 < m \leq X} g(m)\right|^2 = o(B_g(X)^2),
\end{equation}
whenever $10 \leq h \leq X/10$ is an integer that satisfies $h = h(X) \ra \infty$ as $X \ra \infty$.
We begin by making a reduction.
\begin{lem} \label{lem:L2Red}
Suppose that Theorem \ref{thm:MRL2} holds for any non-negative, strongly additive
function $g \in \mc{A}_s$. Then Theorem \ref{thm:MRL2} holds for any $g \in \mc{A}_s$. 
\end{lem}
\begin{proof}
By splitting $g = \text{Re}(g) + i\text{Im}(g)$, and separately decomposing 
\begin{align*}
\text{Re}(g) = \text{Re}(g)^+ - \text{Re}(g)^-, \quad \quad \text{Im}(g) = \text{Im}(g)^+ - \text{Im}(g)^-,
\end{align*}
where, for an additive function $h$ we define the non-negative additive functions $h^{\pm}$ on prime powers via
$$
h^+(p^k) := \max\{h(p^k),0\}, \quad h^-(p^k) := \max\{0,-h(p^k)\},
$$
the Cauchy-Schwarz inequality implies that if \eqref{eq:L2mse} holds for non-negative $g$ satisfying \eqref{eq:LindCond} then it holds for all completely additive $g$ satisfying \eqref{eq:LindCond}. \\
Therefore, we may assume that $g$ is non-negative. Now, by Lemma \ref{lem:redtoSA}, we can find a strongly additive function $g^{\ast}$ such that $B_{G}(X) = o(B_g(X))$ for $G := g-g^{\ast}$.
Thus,
$$
\frac{1}{B_g(X)^2} \sum_{\substack{p \leq X \\ g(p) > \delta^{-1} B_g(X)}} \frac{g(p)^2}{p} \ll \frac{1}{B_{g^{\ast}}(X)^2} \sum_{\substack{p \leq X \\ g^{\ast}(p) > (2\delta)^{-1} B_{g^{\ast}}(X)}} \frac{g^{\ast}(p)^2}{p},
$$
for $X$ large enough. Moreover, we see by the Cauchy-Schwarz inequality and Lemma \ref{lem:TK} that
\begin{align*}
\frac{2}{X}\sum_{X/2<n \leq X} \left|\frac{1}{h}\sum_{n-h < m \leq n} (G(m)-A_G(X))\right|^2 &\leq  \frac{1}{Xh} \sum_{X/2 < n \leq X} \sum_{n-h < m \leq n} |G(m)-A_G(X)|^2 \\ &\ll \frac{1}{X} \sum_{X/3 < m \leq X} |G(m)-A_G(X)|^2 \ll B_G(X)^2 = o(B_g(X)^2).
\end{align*}
Using the estimate $\frac{2}{X}\sum_{X/2 < n \leq X} G(n) = A_G(X) + o(B_g(X))$ by Lemma \ref{lem:mean}, we see that
$$
\Delta_g(X,h) \ll \Delta_{g^{\ast}}(X,h) + \Delta_{G}(X,h) = \Delta_{g^{\ast}}(X,h) + o(B_g(X)),
$$
so that if \eqref{eq:L2mse} holds for strongly additive $g^{\ast} \in \mc{A}_s$ then it also holds for all $g \in \mc{A}_s$. This completes the proof.
\end{proof}

Until further notice we may thus assume that $g$ is strongly additive. For a fixed small parameter $\e> 0$, let $\delta \in (0,1/100)$ be chosen such that $F_g(\delta) < \e$. Let $X$ be a scale chosen sufficiently large so that 
\begin{equation}\label{eq:ScaleChoice}
\sum_{\ss{p \leq X \\ |g(p)| > \delta^{-1} B_g(X)}} \frac{|g(p)|^2}{p} \leq 2F_g(\delta) B_g(X)^2.
\end{equation}
With this data, define
$$
\mc{C} = \mc{C}(X,\delta) := \{p \leq X : |g(p)| \leq \delta^{-1}B_g(X)\}.
$$
We decompose $g$ as
$$
g = g_{\mc{C}} + g_{\mc{P} \bk \mc{C}},
$$
where $g_{\mc{C}}$ and $g_{\mc{P} \bk \mc{C}}$ are strongly additive functions defined at primes by
$$
g_{\mc{C}}(p) := \begin{cases} g(p) &\text{ if $p \in \mc{C}$} \\ 0 &\text{ if $p \notin \mc{C}$.}\end{cases} \quad \quad 
g_{\mc{P}\bk \mc{C}}(p) := \begin{cases} 0 &\text{ if $p \in \mc{C}$} \\ g(p) &\text{ if $p \notin \mc{C}$.}\end{cases}
$$
We will consider the mean-squared errors
$$
\frac{2}{X}\sum_{X/2 < n \leq X} \left|\frac{1}{h} \sum_{n-h<m\leq n} g_{\mc{A}}(m) - \frac{2}{X}\sum_{X/2 < m\leq X} g_{\mc{A}}(m)\right|^2,
$$
for $\mc{A} \in \{\mc{C},\mc{P}\bk \mc{C}\}$, separately. 

\begin{lem} \label{lem:L2IgnoreLrg}
Let $g \in \mc{A}_s$ be a non-negative, strongly additive function. Assume that $X$ and $\delta$ are chosen such that \eqref{eq:ScaleChoice} holds. Then we have
$$
\Delta_g(X) \ll \Delta_{g_{\mc{C}}}(X) + \e B_g(X)^2,
$$
where $\mc{C} = \mc{C}(X,\delta)$.
\end{lem}
\begin{proof}
Arguing as in the proof of Lemma \ref{lem:L2Red}, we see that
$$
\Delta_{g_{\mc{P} \bk \mc{C}}}(X) \ll B_{g_{\mc{P} \bk \mc{C}}}(X)^2 \leq 2 F_g(\delta)B_g(X)^2 < 2\e B_g(X)^2.
$$
Thus, by the Cauchy-Schwarz inequality we obtain
$$
\Delta_g(X) \ll \Delta_{g_{\mc{C}}}(X) + \Delta_{g_{\mc{P}\bk\mc{C}}}(X) \ll \Delta_{g_{\mc{C}}}(X) + \e B_g(X)^2,
$$
as claimed.
\end{proof}

In the present context, given a multiplicative function $f: \mb{N} \ra \mb{C}$, such that $|f(p)|$ is uniformly bounded for all $p \leq X$, we define the following variant of the pretentious distance:
$$
\rho(f,n^{it};X)^2 := \sum_{p \leq X} \frac{|f(p)| - \text{Re}(f(p)p^{-it})}{p}.
$$
We let $t_0 = t_0(f,X)$ denote a real number $t \in [-X,X]$ that minimizes $t \mapsto \rho(f,n^{it};X)^2$.
In the sequel, we work with multiplicative functions determined by $g_{\mc{C}}$, which we define as follows. \\
Fix $r \in (0,\delta^{2}]$. Given $z \in \mb{C}$ satisfing $|z-1| = r$, define 
$$
F_{z}(n) := z^{g_{\mc{C}}(n)/B_g(X)} \text{ for all } n \in \mb{N}.
$$
Since $g_{\mc{C}}$ is strongly additive and satisfies $0 \leq g_{\mc{C}}(p) \leq \delta^{-1} B_g(X)$ for all $p \leq X$, we have
$$
|F_{z}(p^k)| = |F_z(p)| \leq (1+\delta^2)^{\delta^{-1}} \leq e
$$
for all $p^k \leq X$, and thus also
$$
|F_{z}(n)| \leq \left(\max_{p|n} |F_{z}(p)|\right)^{\omega(n)} \leq n^{O\left(\frac{1}{\log\log n}\right)} \ll_{\e} n^{\e},
$$
for any $n \leq X$. Furthermore, as $\delta \in (0,1/100)$, for any $2 \leq u \leq v \leq X$ we get
$$
\sum_{u < p \leq v} \frac{|F_z(p)|}{p} \geq (1-\delta^2)^{\delta^{-1}} \sum_{u < p \leq v} \frac{1}{p} \geq 0.99 \sum_{z < p \leq w} \frac{1}{p}.
$$
In preparation to apply a result from the recent paper \cite{ManMR}, we introduce some further notation. Given a multiplicative function $f: \mb{N} \ra \mb{C}$ set
$$
H(f;X) := \prod_{p \leq X} \left(1+\frac{(|f(p)|-1)^2}{p}\right), \quad\quad \mc{P}_f(X) := \prod_{p \leq X} \left(1+ \frac{|f(p)|-1}{p}\right).
$$
For $B \geq 1$, write $d_B(n)$ to denote the generalized $B$-fold divisor function, i.e., the non-negative multiplicative function arising as the coefficients of the Dirichlet series
$$
\zeta(s)^B = \sum_{n \geq 1} \frac{d_B(n)}{n^s}, \quad\quad \text{Re}(s) > 1.
$$
When $B \in \mb{N}$ these are the usual divisor functions, e.g., $d_2(n) = d(n)$; if $B = 1$ then $d_B \equiv 1$. In general, $d_B(p^{\nu}) = \binom{B+\nu-1}{\nu}$ for any prime power $p^{\nu}$, and in particular $d_B(p^{\nu}) \geq B = d_B(p)$ for all $p$ and $\nu \geq 1$.
\begin{thm}[\cite{ManMR}, Thm. 2.1]\label{thm:MRDB}
Let $B \geq 1$ and  $0 < A \leq B$, and let $X$ be large. Let $f: \mb{N} \ra \mb{C}$ be a multiplicative function that satisfies: \\
\begin{enumerate}[(i)]
\item $|f(n)| \leq d_B(n)$ for all $n \leq X$, and in particular $|f(p)| \leq B$ for all $p \leq X$,  
\item for any $2 \leq u \leq v \leq X$,
$$
\sum_{u < p \leq v} \frac{|f(p)|}{p} \geq A \sum_{u < p \leq v} \frac{1}{p} - O\left(\frac{1}{\log u}\right).
$$
\end{enumerate}
Let $10 \leq h_0 \leq X/(10H(f;X))$, and put $h_1 := h_0 H(f;X)$ and $t_0 = t_0(f,X)$. Then there are constants $c_1,c_2 \in (0,1/3)$, depending only on $A,B$, such that if $X/(\log X)^{c_1} < h_2 \leq X$,
\begin{align*}
&\frac{2}{X}\int_{X/2}^X \left|\frac{1}{h_1}\sum_{x-h_1 < n \leq x} f(n) - \frac{1}{h_1}\int_{x-h_1}^x u^{it_0} du \cdot \frac{2}{X}\sum_{x-h_2 < n \leq x} f(n)\right|^2 dx \\
&\ll_{A,B} \left(\left(\frac{\log\log h_0}{\log h_0}\right)^A + \left(\frac{\log\log X}{(\log X)^{c_2}}\right)^{\min\{1,A\}} \right) \mc{P}_f(X)^2.
\end{align*}
\end{thm}
The conditions (i) and (ii) of the theorem were verified above for $f = F_z$, and it remains to elucidate information about $t_0(F_z,X)$, $H(F_z;X)$ and the size of the Euler product $\mc{P}_{F_z}(X)$.
\begin{lem} \label{lem:removet0}
Fix $r \in (0,\delta^2]$ and let $z \in \mb{C}$ satisfy $|z-1| = r$. Then:
\begin{enumerate}[(a)]
\item $t_0(F_z,X) \ll 1/\log X$,
\item $H(F_z;X) \asymp 1$, and
\item $\mc{P}_{F_z}(X)^2 \ll \prod_{p \leq X} \left(1+p^{-1}\left(|z|^{2g_{\mc{C}}(p)/B_g(X)}-1\right)\right)$.
\end{enumerate}
\end{lem}
\begin{proof}
(a) Applying \cite[(7)]{ManMR} with $A = 0.99$, $B = e$ and $C = 1$ (which is a straightforward consequence of \cite[Lem. 5.1(i)]{MRII}), we see that if $(\log X)|t_0(F_z;X)| \geq D$ for a suitably large constant $D > 0$ then
$$
\rho(F_z,1;X)^2 \geq \sigma \min\{\log\log X, 3\log(|t_0|\log X + 1)\} + O(1) \geq 100,
$$
say, where $\sigma > 0$ is an absolute constant. On the other hand, since $0 \leq 1-\cos x \leq x^2/2$ for all $x \geq 0$, observe that for any $z = re(\theta)$, with $\theta \in [-\pi,\pi]$, we have
$$
\rho(F_z,1;X)^2 = \sum_{p \leq X} |z|^{g_{\mc{C}}(p)/B_g(X)} \frac{1-\cos(\theta g(p)/B_g(X))}{p} \leq \frac{e\pi^2}{2B_g(X)^2} \sum_{p \leq X} \frac{g(p)^2}{p} \leq e\pi^2/2.
$$
This contradiction implies the claim.\\
(b) By Taylor expansion, $|z|^{g_{\mc{C}}(p)/B_g(X)} = 1 + O(\delta g(p)/B_g(X))$, and thus
$$
H(F_z;X) \ll \exp\left(\sum_{p \leq X} \frac{(|z|^{g_{\mc{C}}(p)/B_g(X)}-1)^2}{p}\right) \ll \exp\left(O\left(\frac{\delta^2}{B_g(X)^2}\sum_{p \leq X} \frac{g(p)^2}{p}\right)\right) \ll 1.
$$
(c)  This follows immediately from the upper bounds
$$
\mc{P}_f(X)^2 \ll_B \prod_{p \leq X} \left(1+\frac{2(|f(p)|-1)}{p}\right) \leq \prod_{p \leq X} \left(1+\frac{|f(p)|^2-1}{p}\right),
$$
which are valid whenever $|f(p)| \leq B$ for all $p \leq X$.
\end{proof}

Let $c_1 \in (0,1/3)$ be the constant from Theorem \ref{thm:MRDB}. By Lemma \ref{lem:longH}, if $h_2 = \left\lceil X/(\log X)^{c_1}\right\rceil$ then for any $x \in (X/2,X]$
\begin{equation}\label{eq:swapXforh2}
\frac{1}{h_2}\sum_{x-h_2 < n \leq x} g_{\mc{C}}(n) = \frac{1}{X} \sum_{X/2<n \leq X} g_{\mc{C}}(n) + O(B_g(X)/(\log X)^{1/6}),
\end{equation}
so that it suffices to show that
$$
\Delta_{g_{\mc{C}}}(X;h_1,h_2) := \frac{2}{X}\sum_{X/2 < n \leq X} \left|\frac{1}{h_1}\sum_{n-h_1<m \leq x} g_{\mc{C}}(m) - \frac{1}{h_2} \sum_{n-h_2< m \leq n} g_{\mc{C}}(m)\right|^2 = o(B_g(X)^2),
$$
where $h_1 = h$ and $h_2 = \left\lceil X/(\log X)^{c_1} \right\rceil$.

\begin{lem}\label{lem:L2passtoMult}
Let $g$ be non-negative and strongly additive. Let $r \in (0,\delta^2]$ as above. Then there is $z_0 \in \mb{C}$ with $|z_0-1| = r$ such that
\begin{align*}
&\Delta_{g_{\mc{C}}}(X;h_1,h_2) \\
&\ll \frac{B_g(X)^2}{r^2} |z_0|^{-2\frac{A_g(X)}{B_g(X)}} \frac{2}{X} \sum_{X/2 < n \leq X} \left|\frac{1}{h_1} \sum_{n-h_1 < m \leq n} z_0^{\tilde{g}_{\mc{C}}(m)} - \frac{1}{h_1}\int_{n-h_1}^n u^{it_0} du \cdot \frac{1}{h_2}\sum_{n-h_2 < m \leq n} z_0^{\tilde{g}_{\mc{C}}(m)}m^{-it_0}\right|^2\\
&+\frac{B_g(X)^2}{r^2} (\log X)^{-c_1+o(1)} .
\end{align*}
where $\tilde{g}_{\mc{C}}(m) := g_{\mc{C}}(m)/B_g(X)$, and $t_0 = t_0(F_{z_0},X)$.
\end{lem}
\begin{proof}
For each $n \in (X/2,X]$, $z \in \mb{C}$ and $j = 1,2$, define the maps
\begin{align*}
\phi_n(z;h_j) := \frac{1}{h_j}\sum_{n-h_j < m \leq n} z^{(g_{\mc{C}}(m)-A_{g_{\mc{C}}}(X))/B_g(X)}.
\end{align*}
Note that
$$
\frac{1}{h_j} \sum_{n-h_j < m \leq n} \left(\frac{g_{\mc{C}}(m)-A_{g_{\mc{C}}}(X)}{B_g(X)}\right) = \frac{d}{dz}\phi_n(z;h_j) \Big\rvert_{z = 1}.
$$
Recall that $h_1,h_2 \in \mb{Z}$. Thus, by Cauchy's integral formula we have
\begin{align*}
&\frac{1}{h_1}\sum_{n-h_1 < m \leq n} g_{\mc{C}}(m) - \frac{1}{h_2} \sum_{n-h_2 <m \leq n} g_{\mc{C}}(m) \\
&= \frac{1}{h_1}\sum_{n-h_1 < m \leq n} (g_{\mc{C}}(m)-A_{g_{\mc{C}}}(X)) - \frac{1}{h_2} \sum_{n-h_2 <m \leq n} (g_{\mc{C}}(m)-A_{g_{\mc{C}}}(X)) \\
&= \frac{B_g(X)}{2\pi i} \int_{|z-1| = r} (\phi_n(z;h_1) - \phi_n(z;h_2))\frac{dz}{z^2}.
\end{align*}
By Cauchy-Schwarz, we obtain
\begin{align*}
\Delta_{g_{\mc{C}}}(X;h_1,h_2) &\ll \frac{B_g(X)^2}{r^2} \max_{|z-1| = r} \frac{2}{X} \sum_{X/2 < n \leq X} \left|\phi_n(z;h_1) - \phi_n(z;h_2)\right|^2  \\
&= \frac{B_g(X)^2}{r^2} |z_0|^{-2\frac{A_{g_{\mc{C}}}(X)}{B_g(X)}} \frac{2}{X} \sum_{X/2<n \leq X} \left|\frac{1}{h_1} \sum_{n-h_1 < m \leq n} z_0^{\tilde{g}_{\mc{C}}(m)} - \frac{1}{h_2} \sum_{n-h_2 < m \leq n} z_0^{\tilde{g}_{\mc{C}}(m)}\right|^2,
\end{align*}
for some $z_0 \in \mb{C}$ with $|z_0-1| = r$. To complete the proof, note that by Taylor expansion and Lemma \ref{lem:removet0}(a),
$$
\frac{1}{h_1}\int_{n-h_1}^n u^{it_0} du = n^{it_0} \frac{1-(1-h_1/n)^{1+it_0}}{(1+it_0)h_1/n} = n^{it_0} + O(h_1/X),
$$
and also $m^{-it_0} = n^{-it_0} + O(|t_0|\log(n/m)) = n^{-it_0} + O(h_2/X)$ uniformly in $n-h_2 < m \leq n$. It follows that
$$
\frac{1}{h_2}\sum_{n-h_2 < m \leq n} z_0^{\tilde{g}_{\mc{C}}(m)} = \frac{1}{h_1}\int_{n-h_1}^n u^{it_0} du \cdot \frac{1}{h_2}\sum_{n-h_2 < m \leq n} z_0^{\tilde{g}_{\mc{C}}(m)} m^{-it_0} + O\left(\frac{1}{X}\sum_{n-h_2 < m \leq n} |z_0|^{\tilde{g}_{\mc{C}}(m)}\right).
$$
The error term is, by Shiu's theorem \cite[Thm. 1]{Shiu} and the Cauchy-Schwarz inequality,
$$
\ll \frac{h_2}{X} \exp\left(\sum_{p \leq X} \frac{|z_0|^{\tilde{g}_{\mc{C}}(p)}-1}{p}\right) \ll \frac{h_2}{X} \exp\left(\frac{r}{B_g(X)}\sum_{p \leq X} \frac{g(p)}{p}\right) \ll \frac{h_2}{X} \exp\left(r\sqrt{\log\log X}\right),
$$
which suffices to prove the claim.
\end{proof}

We are now in a position to apply Theorem \ref{thm:MRDB}.
\begin{cor}\label{cor:MRBDtoAdd}
Let $10 \leq h_1 \leq X/10$ be an integer and $h_2 := \lceil X/(\log X)^{c_1}\rceil$ as above. Then there is a constant $\gamma > 0$ such that
$$
\Delta_{g_{\mc{C}}}(X;h_1,h_2) \ll \delta^{-4}\left(\left(\frac{\log\log h}{\log h}\right)^{0.99} + (\log X)^{-\gamma}\right) B_g(X)^2.
$$
\end{cor}
\begin{proof}
Let $z_0$ be chosen as in Lemma \ref{lem:L2passtoMult}. Since $h_1,h_2 \in \mb{Z}$ we may replace the discrete average in Lemma \ref{lem:L2passtoMult} by an integral average at the cost of an error term of size
$$
\ll \max_{x \in [X/2,X]} \left(\frac{1}{h_1} \left|\int_{x-h_1}^x u^{it_0} du - \int_{\llf x \rrf -h_1}^{\llf x \rrf} u^{it_0} du\right| \cdot \frac{1}{h_2} \sum_{\llf x \rrf - h_2 < n \leq \llf x \rrf} |F_{z_0}(n)|\right)^2 \ll \frac{1}{h_1^2} |z_0|^{-2\frac{A_{g_{\mc{C}}}(X)}{B_g(X)}} \mc{P}_{F_{z_0}}(X)^2,
$$
again by \cite[Thm. 1]{Shiu}.
Using the data from Lemma \ref{lem:removet0}, Theorem \ref{thm:MRDB} therefore yields
\begin{equation}\label{eq:penultCor46}
\Delta_{g_{\mc{C}}}(X;h,h_2) \ll \frac{B_g(X)^2}{\delta^{4}} \left(\left(\frac{\log\log h_1}{\log h_1}\right)^{0.99} + \left(\frac{\log\log X}{(\log X)^{c_2}}\right)^{0.99}\right) \cdot |z_0|^{-2\frac{A_{g_{\mc{C}}}(X)}{B_g(X)}} \prod_{p \leq X} \left(1+\frac{|z_0|^{2g_{\mc{C}}(p)/B_g(X)} - 1}{p}\right).
\end{equation}
Put $\rho := \log |z_0| \in (-10\delta^2,10\delta^2)$, say. As $g$ is strongly additive,
$$
A_{g_{\mc{C}}}(X) = \sum_{p \leq X} \frac{g_{\mc{C}}(p)}{p} + O\left(\sum_{p \leq X} \frac{g(p)}{p^2}\right) = \sum_{p \leq X} \frac{g_{\mc{C}}(p)}{p} + O(B_g(X)).
$$
Using the bounds $\log(1+x) \leq x$ and $|e^x-1-x| \leq |x|^2$ for $0 \leq x \leq 1/2$, the rightmost factors on the RHS of \eqref{eq:penultCor46} can thus be estimated as
\begin{align*}
&\ll \exp\left(\sum_{p \leq X} \frac{1}{p}\left(-2\rho\frac{g_{\mc{C}}(p)}{B_g(X)} + \log\left(1+\frac{e^{2\rho g_{\mc{C}}(p)/B_g(X)} - 1}{p}\right)\right)\right)\\
&\leq \exp\left(\sum_{p \leq X} \frac{1}{p}\left(e^{2\rho g_{\mc{C}}(p)/B_g(X)} - 1 - 2\rho \frac{g_{\mc{C}}(p)}{B_g(X)}\right)\right) \leq \exp\left(\frac{4\rho^2}{B_g(X)^2}\sum_{p \leq X} \frac{g_{\mc{C}}(p)^2}{p}\right) \ll 1.
\end{align*}
The claimed bound now follows with any $0 < \gamma < 0.99 c_2$ (changing the implicit constant as needed). 
\end{proof}

\begin{proof}[Proof of Theorem \ref{thm:MRL2}]
Let $g \in \mc{A}_s$. By Lemma \ref{lem:L2Red} we may assume that $g$ is non-negative and strongly-additive. In light of the discussion around \eqref{eq:swapXforh2}, on combining Lemma \ref{lem:L2IgnoreLrg} with Corollary \ref{cor:MRBDtoAdd} we obtain that for any $\e > 0$ there is $\delta >0$ and $X_0 = X_0(\delta)$ such that if $X \geq X_0$ then for any $10 \leq h \leq X/10$, 
$$
\Delta_g(X) \ll \left(\e + \delta^{-4}\left(\left(\frac{\log\log h}{\log h}\right)^{0.99} + (\log X)^{-\gamma}\right)\right)B_g(X)^2.
$$
Selecting $h \geq \exp\left(\delta^{-5} \e^{-2}\log(1/(\delta\e))\right)$, picking $X_0$ larger if necessary, we deduce that $\Delta_g(X) \ll \e B_g(X)$, and the claim follows.
\end{proof}

\section{Gaps and Moments} \label{sec:gaps}
In this section, we will prove Theorem \ref{thm:Elltype}, relating to the moments of the gaps in the sequence $\{g(n)\}_n$. 
\subsection{Small Gaps and Small First Moments are Equivalent: Proof of Theorem \ref{thm:Elltype}(a)}
We start by proving the following quantitative $\ell^1$ gap result.
\begin{prop} \label{prop:convGaps}
Let $0 < \e < 1/3$ and let $X$ be large. Let $g: \mb{N} \ra \mb{C}$ be an additive function. Assume that
$$
\frac{1}{Y} \sum_{n \leq Y} |g(n)-g(n-1)|  \ll \e B_g(X)
$$
for all $X/\log X < Y \leq X$. Then we have
$$
\frac{1}{X} \sum_{n \leq X} |g(n)-A_g(X)| \ll \left(\sqrt{\frac{\log\log(1/\e)}{\log(1/\e)}} + (\log X)^{-1/800}\right) B_g(X).
$$
\end{prop}


\begin{proof}
Let $h = \llf \min\{X/(2\log X), \e^{-1/2}\} \rrf$, and let $X/\log X < Y \leq X$. By the triangle inequality, for any $1 \leq m \leq h$ we obtain
\begin{align*}
\frac{1}{Y} \sum_{h < n \leq Y} |g(n)-g(n-m)| &\leq \frac{1}{Y} \sum_{0 \leq j \leq m-1} \sum_{j < n \leq Y} |g(n-j)-g(n-j-1)|  \\
&\ll \frac{h}{Y} \sum_{1 \leq n \leq Y} |g(n)-g(n-1)| \ll \e^{1/2}B_g(X).
\end{align*}
Averaging over $1 \leq m \leq h$ and then applying the triangle inequality once again, we obtain
$$
\frac{1}{Y} \sum_{h < n \leq Y} \left|g(n)-\frac{1}{h} \sum_{1 \leq m \leq h} g(n-m)\right| \ll \e^{1/2}B_g(X).
$$
Applying Theorem \ref{thm:MRL1},
$$
\frac{1}{Y} \sum_{Y/2<n \leq Y} \left|g(n)-\frac{2}{Y}\sum_{Y/2 < m \leq Y} g(m)\right| \ll B_g(X)\left(\sqrt{\frac{\log\log h}{\log h}} + (\log X)^{-1/800}\right).
$$
For each such $Y$, Lemma \ref{lem:mean} yields
$$
\frac{2}{Y}\sum_{Y/2 < m \leq Y} g(m) = A_g(X) + O\left(\frac{B_g(X)}{\sqrt{\log X}}\right),
$$
and so upon applying Lemma \ref{lem:TK} to $[h,X/\log X]$ and expanding $[X/\log X,X]$ into dyadic segments, we obtain
\begin{align*}
\frac{1}{X}\sum_{h < n \leq X} \left|g(n)-A_g(X)\right| 
&\ll B_g(X)\left(\sqrt{\frac{\log\log h}{\log h}} + (\log X)^{-1/800}\right).
\end{align*}
Reintroducing the segment up to $h$ together with Cauchy-Schwarz and Lemma \ref{lem:TK}, we conclude that
$$
\frac{1}{X}\sum_{n \leq X} |g(n)-A_g(X)| \ll B_g(X)\left( \left(\frac{h}{X}\right)^{1/2} + \sqrt{\frac{\log\log h}{\log h}} + (\log X)^{-1/800}\right),
$$
and so as $h \leq X/\log X$ this implies the claim. 
\end{proof}

\begin{proof}[Proof of Theorem \ref{thm:Elltype}(a)]
By the triangle inequality, we see that if $\frac{1}{X}\sum_{n \leq X} |g(n)-A_g(X)| = o(B_g(X))$ then
$$
\frac{1}{X}\sum_{n \leq X} |g(n)-g(n-1)| \leq \frac{1}{X}\sum_{n \leq X} |g(n)-A_g(X)| + \frac{1}{X}\sum_{m \leq X-1} |g(m)-A_g(X)| = o(B_g(X)).
$$
The converse implication follows immediately from Proposition \ref{prop:convGaps}.
\end{proof}

\subsection{A Conditional Gap Theorem for the Second Moment}
In parallel to the results of the previous subsection, we will apply Theorem \ref{thm:MRL2} to prove the following result. 
\begin{prop} \label{prop:Elltype}
Let $g \in \mc{A}_s$
Then for any integer $10 \leq h \leq X/10$ we have
$$
\frac{1}{X}\sum_{n \leq X} |g(n)-A_g(X)|^2 \ll \frac{h^2}{X}\max_{X/\log X < Y \leq X} \sum_{n \leq Y} |g(n)-g(n-1)|^2 + o_{h \ra \infty}(B_g(X)^2).
$$
\end{prop}

\begin{proof}[Proof of Proposition \ref{prop:Elltype}]
Given our assumptions about $g$, we may apply Theorem \ref{thm:MRL2} to obtain
$$
\sum_{Y/2 < n \leq Y} \left|\frac{1}{h}\sum_{n-h < m \leq n} g(m) - \frac{2}{Y} \sum_{Y/2 < m \leq Y} g(m)\right|^2 = o_{h \ra \infty}(YB_g(X)^2),
$$
for any $X/\log X < Y \leq X$. Applying Lemma \ref{lem:mean}, we deduce that
$$
\frac{1}{X} \sum_{Y/2 < n \leq Y} |g(n)-A_g(Y)|^2 \ll \frac{1}{X} \sum_{Y/2 < n \leq Y} \left|g(n) - \frac{1}{h} \sum_{n-h < m \leq n} g(m)\right|^2 + o_{h \ra \infty}\left(\frac{Y}{X}B_g(X)^2\right).
$$
We of course have
$$
g(n) - \frac{1}{h} \sum_{n-h < m \leq n} g(m) = \frac{1}{h} \sum_{0 \leq j \leq h-1} (g(n)-g(n-j)) = \sum_{0 \leq j \leq h-1} \left(1-\frac{j}{h}\right) (g(n-j)-g(n-j-1)).
$$
Squaring both sides and applying Cauchy-Schwarz, we obtain
\begin{align*}
\frac{1}{X} \sum_{Y/2 < n \leq Y} \left|g(n)-\frac{1}{h}\sum_{n-h < m \leq n} g(m)\right|^2 &\ll \frac{h}{X} \sum_{0\leq j \leq h-1} \sum_{Y/2 < n \leq Y} (g(n-j)-g(n-j-1))^2 \\
&\leq \frac{h^2}{X} \sum_{Y/3 < n \leq Y} (g(n)-g(n-1))^2.
\end{align*}
Combined with the previous estimates, we obtain
$$
\frac{1}{X} \sum_{Y/2 < n \leq Y} |g(n)-A_g(Y)|^2 \ll \frac{h^2}{X} \sum_{n \leq Y} |g(n)-g(n-1)|^2 + o_{h \ra \infty}\left(\frac{Y}{X}B_g(X)^2\right).
$$
We may now complete the proof of the claim by splitting $[1,X]$ into the segments $[1,X/\log X]$ and $[X/\log X,X]$, applying Lemma \ref{lem:TK} trivially to the first segment, and bounding dyadic segments $(Y,2Y] \subset [X/\log X,X]$ using the above arguments. 
\end{proof}
\begin{proof}[Proof of Theorem \ref{thm:Elltype}(b)]
To obtain the theorem, we note first the trivial estimate
$$
\frac{1}{X}\sum_{n \leq X} |g(n)-g(n-1)|^2 \ll \frac{1}{X}\sum_{n \leq X} |g(n)-A_g(X)|^2 + \frac{1}{X}\sum_{m \leq X-1} |g(m)-A_g(X)|^2 \ll \frac{1}{X}\sum_{n \leq X} |g(n)-A_g(X)|^2,
$$
so that if the RHS is $o(B_g(X)^2)$ then so is the LHS. Conversely, suppose that
$$
\frac{1}{Y} \sum_{n \leq Y} |g(n)-g(n-1)|^2 \leq \xi(Y)B_g(Y)^2,
$$
for some function $\xi(Y) \ra 0$, for all $Y$ large enough. Suppose $\max_{X/\log X < Y \leq X} \xi(Y) = \xi(Y_0)$, and put $h := \lfloor \xi(Y_0)^{-1/3}\rfloor.$ 
By Proposition \ref{prop:Elltype},
$$
\frac{1}{X} \sum_{n \leq X} |g(n)-A_g(X)|^2 \ll \xi(Y_0)^{-2/3} \cdot \xi(Y_0) B_g(X)^2 + o(B_g(X)^2) = o(B_g(X)^2)
$$
as $X \ra \infty$, as required.
\end{proof}

\section{Erd\H{o}s' Almost Everywhere Monotonicity Problem}
Let $g: \mb{N} \ra \mb{R}$ be additive. For convenience, set $g(0) := 0$, and recall the definitions 
$$
\mc{B} := \{n \in \mb{N} : g(n) < g(n-1)\}, \quad \quad \mc{B}(X) := \mc{B} \cap [1,X].
$$ 
In this section, we will pursue the study of functions $g$ such that $|\mc{B}(X)| = o(X)$.
\subsection{Small support of variance is equivalent to small prime support}
To prove Theorem \ref{thm:iterStep} we will eventually need control over a sparsely-supported sum such as
$$
\frac{1}{X} \sum_{n \in \mc{B}(X)} |g(n)-A_g(X)|^2,
$$
with the objective of obtaining savings over the trivial bound $O(B_g(X)^2)$ from Lemma \ref{lem:TK}. The purpose of this subsection is to determine sufficient conditions in order to achieve a non-trivial estimate of this kind. \\ 
Given a set of positive integers $\mc{S}$, a positive real number $X \geq 1$ and a prime power $p^k \leq X$, write $\mc{S}(X) := \mc{S} \cap [1,X]$ and $\mc{S}_{p^k}(X) := \{n \in \mc{S}(X) : p^k |n\}$.

\begin{prop}\label{prop:locLg}
Let $g: \mb{N} \ra \mb{C}$ be an additive function belonging to $\mc{A}$. Let $\mc{S}$ be a set of integers with $|\mc{S}(X)| = o(X)$, and let $\e \in (0,1)$ satisfy the conditions
$$
|\mc{S}(X)|/X < \e/2, \quad\quad \sum_{\ss{p^k \leq X \\ k \geq 2}} \frac{|g(p)|^2 + |g(p^k)|^2}{p^k} \leq \e B_g(X)^2.
$$
Then the following bound holds:
$$
\frac{1}{X}\sum_{n \in \mc{S}(X)} |g(n)-A_g(X)|^2 \ll B_g(X)^2\left(\e + \e^{-1} \left(\frac{|\mc{S}(X)|}{X}\right)^{1/2}\right) + \sum_{\ss{p \leq X \\ |\mc{S}_p(X)| > \e X/p}} \frac{|g(p)|^2}{p}.
$$
Moreover, we have
\begin{equation}\label{eq:sparseBp}
\sum_{\ss{p \leq X \\ |\mc{S}_p(X)| > \e X/p}} \frac{1}{p} \ll \e^{-2} \frac{|\mc{S}(X)|}{X}.
\end{equation}
\end{prop}

\begin{rem}
Proposition \ref{prop:locLg} states that if the bulk of the contribution to the variance of $g(n)$ occurs along a sparse subset $\mc{S}(X) \subseteq [1,X]$ then $B_g(X)$ is dominated by primes $p$ of which $\mc{S}(X)$ has many multiples $\leq X$. For sufficiently small primes $p$ this is ruled out by the sparseness of $\mc{S}(X)$, but it may occur for large enough primes. 
\end{rem}
Our proof will proceed by applying variants of the large sieve and Tur\'{a}n-Kubilius inequalities. The first of these is due to Elliott.

\begin{lem}[Elliott's Dual Tur\'{a}n-Kubilius Inequality] \label{lem:DualTK1}
Let $\{a(n)\}_n \subset \mb{C}$ be a sequence and let $X \geq 2$. Then
$$
\sum_{p \leq X} p\left|\sum_{\substack{n \leq X \\ p|n}} a(n) - \frac{1}{p} \sum_{n \leq X} a(n)\right|^2 \ll X \sum_{n\leq X} |a(n)|^2.
$$
\end{lem}
\begin{proof}
This is \cite[Lemma 5.2]{Ell} (taking $\sg = 0$ there).
\end{proof}

A cheaper variant of the latter result, for divisibility by products of two large primes, is as follows. 
\begin{lem}[Variant of Dual Tur\'{a}n-Kubilius] \label{lem:DualTK2}
Let $\{a(n)\}_n \subset \mb{C}$. Then
$$
\sum_{\substack{X^{1/4} < p,q \leq X \\ p \neq q}} pq\left|\sum_{\substack{n \leq X \\ pq|n}} a(n) - \frac{1}{pq} \sum_{n \leq X} a(n)\right|^2 \ll X \sum_{n \leq X} |a(n)|^2.
$$
\end{lem}
\begin{proof}
By including the factor $pq$ into the square, we observe that this establishes an $\ell^2 \ra \ell^2$ operator norm for the matrix with entries 
$$
M = \left((pq)^{1/2}1_{pq|n} - (pq)^{-1/2}\right)_{\substack{X^{1/4} < p,q \leq X, p\neq q \\ n \leq X}}.
$$ 
Thus, by the duality principle \cite[Sec. 7.1]{IK} it suffices to show that for any sequence $\{b(p,q)\}_{p,q \text{ prime}} \subset \mb{C}$ we have
$$
\sum_{n \leq X} \left|\sum_{\substack{X^{1/4} < p,q \leq X \\ p \neq q}} \frac{b(p,q)}{pq} (pq1_{pq|n} - 1)\right|^2 \ll X\sum_{\substack{X^{1/4} < p,q \leq X \\ p \neq q}} \frac{|b(p,q)|^2}{pq}.
$$
Expanding the square on the LHS and swapping orders of summation, we obtain
$$
\sum_{\ss{X^{1/4} < p_1,p_2,q_1,q_2 \leq X \\ p_j \neq q_j \\ j = 1,2}} \frac{b(p_1,q_1)\bar{b}(p_2,q_2)}{p_1q_1p_2q_2} \sum_{n \leq X}\left(p_1q_1p_2q_2 1_{p_1q_1|n}1_{p_2q_2|n} - p_1q_11_{p_1q_1|n} - p_2q_21_{p_2q_2|n} + 1\right).
$$
Fix the quadruple $(p_1,q_1,p_2,q_2)$ for the moment, and consider the inner sum over $n \leq X$. If $(p_1q_1,p_2q_2) = 1$ then as $p_1q_1p_2q_2 > X$ the sum is
$$
\ll p_1q_1\llf \frac{X}{p_1q_1}\rrf + p_2q_2 \llf \frac{X}{p_2q_2} \rrf + X \ll X.
$$
If $(p_1q_1, p_2q_2) = p_1$ (so that $q_1 \neq q_2)$, say, then the sum is
$$
\ll p_1^2q_1q_2 \llf \frac{X}{p_1q_1q_2} \rrf + p_1q_1\llf \frac{X}{p_1q_1}\rrf + p_2q_2 \llf \frac{X}{p_2q_2} \rrf + X \ll p_1 X,
$$
provided $p_1q_1q_2 \leq X$. By symmetry, the analogous result holds if $(p_1q_1,p_2q_2) = q_1$. Finally, if $p_1q_1 = p_2q_2$ then similarly the bound is $\ll p_1q_1 X$. We thus obtain from these cases that the expression is bounded above by
\begin{align*}
&\ll X \sum_{\ss{X^{1/4} < p_1,q_1,p_2,q_2 \leq X \\ p_1 \neq q_1, p_2 \neq q_2 \\ (p_1q_1,p_2q_2) = 1}} \frac{|b(p_1,q_1)||b(p_2,q_2)|}{p_1p_2q_1q_2} + \sum_{\ss{X^{1/4} < p,q,r \leq X \\ p \neq q, p \neq r, q \neq r}} \frac{|b(p,q)||b(p,r)| + |b(p,q)||b(r,q)|}{pqr} \\
&+ X\sum_{\ss{X^{1/4} < p,q \leq X \\ p \neq q}} \frac{|b(p,q)|(|b(p,q)| + |b(q,p)|)}{pq}.
\end{align*}
By the AM-GM inequality we simply have $2|b(p,q)||b(p',q')| \leq |b(p,q)|^2 + |b(p',q')|^2$ for any pairs of primes $p,q$ and $p',q'$, so invoking Mertens' theorem and symmetry the above expressions are
\begin{align*}
\ll X \sum_{\ss{X^{1/4} < p,q \leq X \\ p \neq q}} \frac{|b(p,q)|^2}{pq},
\end{align*}
and the claim follows.
\end{proof}

\begin{proof}[Proof of Proposition \ref{prop:locLg}]
Let $g \in \mc{A}$, and let $g^{\ast}$ be the strongly additive function equal to $g$ at primes. Following the proof of Lemma \ref{lem:redtoSA}, and using Lemma \ref{lem:TK}, we have
$$
\frac{1}{X}\sum_{n \in \mc{S}(X)} |(g-g^{\ast}) - A_{g-g^{\ast}}(X)|^2 \ll B_{g-g^{\ast}}(X)^2 \ll \sum_{\ss{p^k \leq X \\ k \geq 2}} \frac{|g(p)|^2 + |g(p^k)|^2}{p^k} \ll \e B_g(X)^2,
$$
by assumption. 
It follows that
$$
\frac{1}{X}\sum_{n \in \mc{S}(X)} |g(n)-A_g(X)|^2 \ll \e B_g(X)^2 + \frac{1}{X}\sum_{n \in \mc{S}(X)} |g^{\ast}(n)-A_{g^{\ast}}(X)|^2,
$$
so replacing $g$ by $g^{\ast}$, we may assume in what follows that $g$ is strongly additive. \\
Fix $z = X^{1/4}$ and split $g = g_{\leq z} + g_{>z}$, where $g_{\leq z}$ is the strongly additive function supported on primes $p \leq z$. By Cauchy-Schwarz, we seek to estimate
\begin{equation}\label{eq:splitAtZ}
\frac{1}{X} \sum_{n \in \mc{S}(X)} |g_{\leq z}(n) - A_{g_{\leq z}}(X)|^2 + \frac{1}{X} \sum_{n \in \mc{S}(X)} |g_{> z} (n) - A_{g_{ > z}}(X)|^2.
\end{equation}
We begin with the first expression.
Writing
$$
g_{\leq z}(n)-A_{g_{\leq z}}(X) = \sum_{p \leq z} g(p)\sum_{1 \leq k \leq \log X/\log p} \left(1_{p^k||n} - p^{-k}(1-1/p)\right) = \sum_{p \leq z} g(p)(1_{p|n}-1/p) + O(X^{-1} \sum_{p \leq z} |g(p)|)
$$
for each $n \leq X$ and expanding the square, this first term in \eqref{eq:splitAtZ} is
\begin{align*}
&\ll \sum_{p \leq z} \frac{|g(p)|^2}{p^{2}} \frac{1}{X} \sum_{n \in \mc{S}(X)} (p1_{p|n} - 1)^2 + \frac{1}{X}\sum_{n \in \mc{S}(X)}\sum_{\substack{p,q \leq z \\ p \neq q}} \frac{|g(p)g(q)|}{pq} (p1_{p|n} - 1)(q1_{q|n} - 1) + O(B_g(X)^2X^{-2}(z\pi(z))) \\
&=: D + O + O(B_g(X)^2X^{-3/2}).
\end{align*}
Consider the off-diagonal term $O$. 
Observe that for any two distinct primes $p$ and $q$, the Chinese remainder theorem implies that
$$
(p1_{p|n}-1)(q1_{q|n}-1) = \left(\sum_{\substack{a \pmod{p} \\ a \neq 0}} e(an/p)\right)\left(\sum_{\substack{b \pmod{q} \\ b \neq 0}} e(bn/q)\right) = \asum_{c \pmod{pq}} e(cn/pq),
$$
where the asterisked sum is over reduced residues modulo $pq$. Note that for any two distinct products $p_1q_1$ and $p_2q_2$ the gap between fractions with these denominators satisfies 
$$
\left|\frac{c_1}{p_1q_1} - \frac{c_2}{p_2q_2}\right| \geq \frac{1}{p_1q_1p_2q_2} \geq \frac{1}{z^4} = \frac{1}{X},
$$ 
and the number of pairs yielding the same product $pq$ is $\leq 2$. Using this expression in $O$, applying the Cauchy-Schwarz inequality twice followed by the large sieve inequality \cite[Lem. 7.11]{IK}, we obtain
\begin{align*}
O &\leq \frac{1}{X}\sum_{\ss{p,q \leq z \\ p \neq q}} \frac{|g(p)g(q)|}{pq} \asum_{c \pmod{pq}} \sum_{n \leq X} 1_{\mc{S}}(n) e(cn/pq) \\
&\leq \frac{1}{X}\left(\sum_{p \leq z} \frac{|g(p)|^2}{p}\right)\left(\sum_{\substack{pq \leq \sqrt{X} \\ p \neq q}} \frac{1}{pq} \left|\asum_{c \pmod{pq}} \sum_{n \leq X} 1_{\mc{S}}(n) e(cn/pq)\right|^2\right)^{1/2} \\
&\ll \frac{1}{X}B_g(X)^2 \left(\sum_{\substack{pq \leq \sqrt{X} \atop p\neq q}}\asum_{c \pmod{pq}}  \left|\sum_{n \leq X} 1_{\mc{S}}(n)e(cn/pq)\right|^2\right)^{1/2} \\
&\ll \frac{1}{X}B_g(X)^2 \left(X|\mc{S}(X)|\right)^{1/2} = \left(\frac{|\mc{S}(X)|}{X}\right)^{1/2}B_g(X)^2.
\end{align*}
Thus, from \eqref{eq:splitAtZ} we obtain the upper bound,
\begin{align*}
&\ll \frac{1}{X}\sum_{p \leq z} \frac{|g(p)|^2}{p^2}\sum_{n \in \mc{S}(X)} |p1_{p|n}-1|^2 + \frac{1}{X} \sum_{n \in \mc{S}(X)} \left|\sum_{X^{1/4} < p \leq X} \frac{|g(p)|}{p}(p1_{p|n}-1)\right|^2 + \left(\frac{|\mc{S}(X)|}{X}\right)^{1/2} B_g(X)^2\\
&\ll \sum_{p \leq X} \frac{|g(p)|^2}{p^2}\frac{1}{X}\sum_{n \in \mc{S}(X)}|p1_{p|n}-1|^2 + \sum_{\ss{X^{1/4} < p,q \leq X \\ p \neq q}} \frac{|g(p)g(q)|}{pq} \frac{1}{X}\sum_{n \in \mc{S}(X)}(p1_{p|n}-1)(q1_{q|n}-1) \\
&+ \left(\frac{|\mc{S}(X)|}{X}\right)^{1/2} B_g(X)^2.
\end{align*}
Call the second term above $T$, so that
\begin{align*}
T &= \sum_{\ss{X^{1/4} < p,q \leq X \\ p \neq q}} \frac{|g(p)g(q)|}{pq} \left(\frac{pq}{X}|\mc{S}_{pq}(X)| - \frac{p}{X} |\mc{S}_p(X)| - \frac{q}{X} |\mc{S}_q(X)| + \frac{|\mc{S}(X)|}{X}\right) \\
&=: \sum_{\ss{X^{1/4} < p,q \leq X \\ p \neq q}} \frac{|g(p)g(q)|}{pq} T_{p,q}(X). 
\end{align*}
 We split the pairs of primes $X^{1/4} < p,q \leq X$, $p \neq q$ in its support as follows. Given a squarefree integer $d$, call $E_d(\e)$ the condition $\frac{d}{X}|\mc{S}_d(X)| \leq \e$, and let $L_d(\e)$ be the converse condition $\frac{d}{X}|\mc{S}_d(X)| > \e$. If, simultaneously, the three conditions $E_{pq}(\e),E_p(\e)$ and $E_q(\e)$ all hold, then as $|\mc{S}(X)|/X < \e$ we have $T_{p,q}(X) \ll \e$; otherwise, we trivially have $T_{p,q}(X) \ll 1$. We thus find by the Cauchy-Schwarz inequality that
\begin{align}
T &\ll \e \sum_{\ss{X^{1/4} < p,q \leq X \\ p \neq q}} \frac{|g(p)g(q)|}{pq} + \sum_{\ss{X^{1/4} < p,q \leq X \\ p \neq q \\ L_p(\e), L_q(\e) \text{ or } L_{pq}(\e)}} \frac{|g(p)g(q)|}{pq} \nonumber\\
&\ll B_g(X)^2 \left( \e \sum_{X^{1/4} < p \leq X} \frac{1}{p} + \left(\sum_{\ss{X^{1/4} < p,q \leq X \\ p \neq q \\ L_p(\e), L_q(\e) \text{ or } L_{pq}(\e)}} \frac{1}{pq}\right)^{1/2}\right). \label{eq:lastOD}
\end{align}
Now suppose $L_d(\e)$ holds for some $d\geq 2$. As $\e > 2 \frac{|\mc{S}(X)|}{X}$ we have 
$$
\frac{1}{d} = \frac{4d}{(\e X)^2} \left(\frac{\e X}{2d}\right)^2 < \frac{4d}{(\e X)^2} \left(|\mc{S}_d(X)| - \frac{|\mc{S}(X)|}{d}\right)^2.
$$
As such, an application of Lemma \ref{lem:DualTK1} yields
$$
\sum_{\ss{X^{1/4} < p \leq X \\ L_p(\e)}} \frac{1}{p} \leq \sum_{\ss{p \leq X \\ |\mc{S}_p(X)| > \e X/p}} \frac{1}{p} \ll \frac{1}{(\e X)^{2}} \sum_{p \leq X} p \left|\sum_{\ss{n \leq X \\ p | n}} 1_{\mc{S}}(n) - \frac{1}{p} \sum_{n \leq X} 1_{\mc{S}}(n)\right|^2 \ll \e^{-2} \frac{|\mc{S}(X)|}{X};
$$
this, by the way, establishes \eqref{eq:sparseBp}. Similarly, by Lemma \ref{lem:DualTK2} we get
$$
\sum_{\ss{X^{1/4} < p,q \leq X \\ p \neq q \\ L_{pq}(\e)}} \frac{1}{pq} \ll \frac{1}{(\e X)^{2}} \sum_{\ss{X^{1/4} < p,q \leq X \\ p\neq q}} pq \left|\sum_{\ss{ n \leq X \\ pq|n}} 1_{\mc{S}}(n) - \frac{1}{pq} \sum_{n \leq X} 1_{\mc{S}}(n)\right|^2 \ll \e^{-2} \frac{|\mc{S}(X)|}{X}.
$$
Combining these estimates in \eqref{eq:lastOD} shows that 
$$
T \ll B_g(X)^2 \left(\e + \e^{-1} \left(\frac{|\mc{S}(X)|}{X}\right)^{1/2}\right).
$$
Thus, putting $\delta(X) := \e + \e^{-1} \left(\frac{|\mc{S}(X)|}{X}\right)^{1/2}$, we finally conclude that 
\begin{align*}
\frac{1}{X} \sum_{n \in \mc{S}(X)} |g(n)-A_g(X)|^2 &\ll \sum_{p \leq X} \frac{|g(p)|^2}{p^2} \frac{1}{X}\sum_{n \in \mc{S}(X)}|p1_{p|n}-1|^2 +  \delta(X)B_g(X)^2 \\
&= \sum_{p \leq X} \frac{|g(p)|^2}{p} \left(\frac{(p-2)|\mc{S}_p(X)|}{X} + \frac{|\mc{S}(X)|}{pX}\right) + \delta(X) B_g(X)^2\\
&\ll \sum_{\ss{p \leq X \\ |\mc{S}_p(X)| > \e X/p}} \frac{|g(p)|^2}{p} + B_g(X)^2\left(\e + \frac{|\mc{S}(X)|}{X}\right) + \delta(X) B_g(X)^2\\
&\ll \sum_{\ss{p \leq X \\ |\mc{S}_p(X)| > \e X/p}} \frac{|g(p)|^2}{p} + \delta(X) B_g(X)^2,
\end{align*}
and the claim follows. 
\end{proof}

\begin{cor} \label{cor:sparseLind}
Let $g: \mb{N} \ra \mb{R}$ be an additive function that satisfies \eqref{eq:LindCond}. Let $\mc{S} \subset \mb{N}$ be a set with $|\mc{S}(X)| = o(X)$. Then for any $j \in \mb{Z}$,
$$
\frac{1}{X} \sum_{n+j \in \mc{S}(X)} |g(n)-A_g(X)|^2 = o(B_g(X)^2).
$$
\end{cor}
\begin{proof}
Fix $j \in \mb{Z}$. Since $|(\mc{S}-j)(X)| = |\mc{S}(X)| = o(X)$, by Proposition \ref{prop:locLg} it suffices to show that
$$
\sum_{\substack{p \leq X \\ |(\mc{S}-j)_p(X)| > \e X/p}} \frac{g(p)^2}{p} \ll \e B_g(X)^2,
$$
for any $\e > 0$ sufficiently small, as $X \ra \infty$. \\
We may split the sum according to whether or not $|g(p)| > \delta^{-1}B_g(X)$, where $\delta > 0$ is to be chosen. In light of \eqref{eq:sparseBp}, we obtain
$$
\sum_{\substack{p \leq X \\ |(\mc{S}-j)_p(X)| > \e X/p \\ |g(p)| \leq \delta^{-1} B_g(X)}} \frac{g(p)^2}{p} \leq \delta^{-2} B_g(X)^2 \sum_{\substack{p \leq X \\ |(\mc{S}-j)_p(X)| > \e X/p}} \frac{1}{p} \ll (\e\delta)^{-2} \frac{|\mc{S}(X)|}{X} B_g(X)^2,
$$
so that this is $\ll \e B_g(X)^2$ if $X \geq X_0(\e)$. \\
On the other hand, by our assumption \eqref{eq:LindCond}, 
$$
\sum_{\substack{p \leq X \\ |(\mc{S}-j)_p(X)| > \e X/p \\ |g(p)| > \delta^{-1} B_g(X)}} \frac{g(p)^2}{p} \leq \sum_{\substack{p \leq X \\ |g(p)| > \delta^{-1} B_g(X)}} \frac{g(p)^2}{p} \leq 2F_g(\delta)B_g(X)^2,
$$
provided $X \geq X_0(\delta)$. For $\delta = \delta(\epsilon)$ sufficiently small we can make this $\ll \e B_g(X)^2$ whenever $X\geq X_0(\e)$ (with $X_0(\e)$ taken larger if necessary). The claim now follows.
\end{proof}

We are now able to prove the first part of Theorem \ref{thm:iterStep}, namely that there is a parameter $\lambda = \lambda(X)$ such that
$$
\sum_{p^k \leq X} \frac{|g(p^k)-\lambda(X)\log p^k|^2}{p^k} = o(B_g(X)^2).
$$
The proof of the slow variation condition $\lambda(X^u) = \lambda(X) + o(B_g(X)/\log X)$, for $0 < u \leq 1$ fixed, is postponed to the next section.
\begin{proof}[Proof of Theorem \ref{thm:iterStep}: Part I]
In light of Lemma \ref{lem:Ruzsa}, it begin by showing that
\begin{equation}\label{eq:passtoRuzsa}
\frac{1}{X} \sum_{n \leq X} |g(n)-A_g(X)|^2 = o(B_g(X)^2).
\end{equation}
As in Lemma \ref{lem:Ag}, when $X/\log X < Y \leq X$ we have 
$$
|A_g(X)-A_g(Y)| \leq B_g(X) \left(\sum_{X/\log X < p \leq X} \frac{1}{p}\right)^{1/2} \ll B_g(X) \sqrt{\frac{\log\log X}{\log X}},
$$
so that upon dyadically decomposing the sum on the LHS, we get
\begin{align*}
&\frac{1}{X}\sum_{n \leq X/\log X} |g(n)-A_g(X)|^2 + \sum_{1 \leq 2^j \leq \log X} 2^{-j} \cdot \frac{2^j}{X}\sum_{X/2^j < n \leq 2X/2^j} |g(n)-A_g(X)|^2 \\
&= \sum_{\frac{X}{\log X} \leq 2^j \leq X} \frac{2^j}{X} \cdot 2^{-j} \sum_{2^{j-1} < n \leq 2^j} |g(n)-A_g(2^j)|^2 + O\left(\frac{B_g(X)^2 \log\log X}{\log X}\right).
\end{align*}
It thus suffices to show that, uniformly over $X/\log X < 2^j \leq X$,
$$
2^{-j} \sum_{2^{j-1} < n \leq 2^j} |g(n)-A_g(2^j)|^2 = o(B_g(X)^2).
$$
Fix $X/\log X < 2^k \leq X$, set $Y_k := 2^k$ and introduce a parameter $1 \leq R \leq (\log X)^{1/2}$, which is very slowly-growing as a function of $X$. Let 
$$
\mc{B}_R(Y_k) := \bigcup_{|i| \leq R} (\mc{B} \cap (Y_k/2,Y_k])+i), \quad\quad \mc{G}_R(Y_k) := [Y_k/2,Y_k] \bk \mc{B}_R(Y_k).
$$
We observe that if $n \in \mc{G}_R(Y_k)$ then we have
$$
g(n-R) \leq g(n-R+1) \leq \cdots \leq g(n) \leq g(n+1) \leq \cdots \leq g(n+R).
$$
We divide $\mc{G}_R(Y_k)$ further into the sets
$$
\mc{G}_R^+(Y_k) := \{n \in \mc{G}_R(Y_k) : g(n) \geq A_g(Y_k)\}, \quad \quad \mc{G}_R^-(Y_k) := \{n \in \mc{G}_R(Y_k) : g(n) < A_g(Y_k)\}.
$$
Suppose $n \in \mc{G}_R^+(Y_k)$. Since 
$$
0 \leq g(n)-A_g(Y_k) \leq \frac{1}{R} \sum_{0 \leq j \leq R-1} g(n+j) - A_g(Y_k),
$$
we deduce from the monotonicity of the map $y \mapsto y^2$ for $y \geq 0$ that (shifting $n \mapsto n+R =: n'$)
$$
\frac{2}{Y_k}\sum_{\ss{n \in \mc{G}_R^+(Y_k) \\ n+R \leq Y_k}} |g(n)-A_g(Y_k)|^2 \leq \frac{2}{Y_k}\sum_{\ss{n'-R \in \mc{G}_R^+(Y_k) \\ n' \leq Y_k}} \left|\frac{1}{R} \sum_{n'-R<m \leq n'} g(m) - \frac{2}{Y_k}\sum_{Y_k/2 < m \leq Y_k} g(m)\right|^2 + O\left(\frac{B_g(X)^2}{\log X}\right),
$$
where the error term comes from replacing $A_g(Y_k)$ by the sum over $[Y_k/2,Y_k]$. 
Similarly, if $n \in \mc{G}_R^-(Y_k)$ then
$$
0 \leq A_g(Y_k) - g(n) \leq A_g(Y_k) - \frac{1}{R} \sum_{0 \leq j \leq R-1} g(n-j),
$$
and so by the same argument we obtain
$$
\frac{2}{Y_k} \sum_{\ss{n \in \mc{G}_R^-(Y_k) \\ n-R \geq Y_k/2}} |g(n)-A_g(Y_k)|^2 \leq \frac{2}{Y_k} \sum_{\ss{n \in \mc{G}_R^-(Y_k) \\ n-R \geq Y_k/2}} \left|\frac{1}{R}\sum_{n-R < m \leq n} g(m) - \frac{2}{Y_k}\sum_{Y_k/2 < m \leq Y_k} g(m)\right|^2 + O\left(\frac{B_g(X)^2}{\log X}\right).
$$
The above sums cover all elements of $\mc{G}_R(Y_k)$ besides those in $[Y_k/2,Y_k/2 + R) \cup (Y_k-R,Y_k]$. To deal with these, we define $\mc{S} := \bigcup_{j \geq 0} [2^j-R,2^j+R] \cap \mb{N}$. We see that $|\mc{S}(Z)| \ll R \log Z = o(Z)$, and $\mc{S}$ contains $[Y_k/2,Y_k/2 + R] \cup [Y_k-R,Y_k]$ for each $k$. By Corollary \ref{cor:sparseLind} (taking $j = 0$ there), we thus obtain
$$
\frac{1}{Y_k}\sum_{\ss{ n \in \mc{G}_R(Y_k) \cap \mc{S}(Y_k)}} |g(n)-A_g(Y_k)|^2 = o(B_g(X)^2)
$$
uniformly over all $X/\log X < Y_k \leq X$, provided $X$ is large enough in terms of $R$.
Combining the foregoing estimates and using positivity, we find that
\begin{align*}
\frac{2}{Y_k}\sum_{n \in \mc{G}_R(Y_k)} |g(n)-A_g(Y_k)|^2 &\ll \frac{1}{Y_k}\sum_{Y_k/2 < n \leq Y_k} \left|\frac{1}{R} \sum_{n-R < m \leq n} g(m) - \frac{2}{Y}\sum_{Y_k/2 < m \leq Y_k} g(m)\right|^2 + o(B_g(X)^2).
\end{align*}
By Theorem \ref{thm:MRL2}, this gives $o_{R \ra \infty}(B_g(X)^2)$. \\
It remains to estimate the contribution from $n \in \mc{B}_R(Y_k)$. By the union bound, we have
$$
\frac{2}{Y_k} \sum_{n \in \mc{B}_R(Y_k)} |g(n)-A_g(Y_k)|^2 \leq R\max_{|i| \leq R} \frac{2}{Y_k}\sum_{\substack{Y_k/2 < n \leq Y_k \\ n +i \in \mc{B}}} |g(n)-A_g(Y_k)|^2.
$$
By Corollary \ref{cor:sparseLind}, the above expression is $o(B_g(X))$, again provided $X$ is sufficiently large in terms of $R$. \\
To conclude, for any $\e > 0$ we can find $R$ large enough in terms of $\e$ and $X_0$ sufficiently large in terms of $\e$ and $R$ such that if $X \geq X_0$ then
$$
\frac{2}{Y_k} \sum_{Y_k/2 < n \leq Y_k} |g(n)-A_g(Y_k)|^2 \ll \e B_g(X)^2
$$
uniformly in $X/\log X < Y_k = 2^k \leq X$, and \eqref{eq:passtoRuzsa} follows.\\
Now, applying Lemma \ref{lem:Ruzsa}, we deduce that
$$
B_{g_{\lambda_0}}(X)^2 + \lambda_0(X)^2 \ll \frac{1}{X}\sum_{n \leq X} |g(n)-A_g(X)|^2 = o(B_g(X)^2),
$$
where $g_{\lambda_0}(n) := g(n)-\lambda_0 \log n$, and 
$$
\lambda_0(X) := \frac{2}{(\log X)^2} \sum_{p \leq X} \frac{g(p)\log p}{p}.
$$
We verify that $\lambda_0(X)$ is slowly varying in the next section (immediately following the proof of Proposition \ref{prop:AgLog}).
\end{proof}

\section{Rigidity Properties for Almost Everywhere Monotone Functions}
We continue to assume that $g$ is almost everywhere monotone. Theorem \ref{thm:iterStep} shows how an additive function $g \in \mc{A}$ compares to a logarithm, conditionally assuming $g(p)$ is not frequently much larger than $B_g(X)$ for $p \leq X$. In this section, we endeavour to explore some further conditional and unconditional consequences of this almost everywhere monotonicity property.  
\subsection{The structure of $A_g(X)$}
The first main result of this section shows that the asymptotic mean $A_g(X)$ of $g$ behaves similarly to a constant times $\log$, in the sense that for $\delta \in (0,1)$ fixed and $X^{\delta} < s,t \leq X$, the quotient $(A_g(t)-A_g(s))/\log (t/s)$ does not vary much.
\begin{prop} \label{prop:AgLog}
Let $g:\mb{N} \ra \mb{C}$ be an additive function, and assume that $\mc{B} := \{n \in \mb{N} : g(n) < g(n-1)\}$ has natural density 0. Then there is $\lambda = \lambda(X) \in \mb{R}$ such that for any $\frac{\log\log X}{\sqrt{\log X}} < \delta \leq 1/4$,
$$
\sum_{X^{\delta} < p^k \leq X} \frac{1}{p^k}|A_g(X)-A_g(X/p^k) - \lambda \log p^k| = o(B_g(X)(\log(1/\delta))^{1/2}).
$$
Furthermore, $A_g(X)$ and $\lambda(X)$ satisfy the following properties: 
\begin{enumerate}[(i)]
\item $\lambda(X) \ll B_g(X)/\log X$, \\
\item for $X$ sufficiently large and any $X^{\delta} < t_1 \leq t_2 \leq X$,
$$
A_g(t_2) = A_g(t_1) + \lambda(X)\log(t_1/t_2) + o((\log(1/\delta)^{1/2}B_g(X)),
$$
\item for every $u \in (\delta,1]$ we have
$$
\lambda(X) = \lambda(X^u) + o\left((\log(1/\delta)^{1/2} \delta^{-1} \frac{B_g(X)}{\log X}\right).
$$
\end{enumerate}
\end{prop}
\begin{rem}
We would like to determine $A_g(t)$ directly as a function of $t$ in some range, say $X^{\delta} < t \leq X$. Proposition \ref{prop:AgLog} provides the approximation $A_g(t) = A_g(X^{\delta}) + (1-\delta u) \lambda(X) \log t + o(B_g(X))$, where $u := \log X/\log t$, but this still contains a reference to a second value $A_g(X^{\delta})$. We might iterate this argument to obtain (using the slow variation of $\lambda$) a further approximation in terms of $A_g(X^{\delta^2})$, $A_g(X^{\delta^3})$, and so forth, but without further data about $g$ (say, $A_g(X^{1/1000}) = o(B_g(X))$) it is not obvious that this argument yields an asymptotic formula for $A_g(t)$ alone. 
\end{rem}

To prove this we will require a few lemmas.
\begin{lem} \label{lem:EllLS}
Let $g: \mb{N} \ra \mb{C}$ be an additive function. Let $\alpha \in (1,2)$ and $\frac{\log\log X}{\sqrt{\log X}} < \delta \leq 1/4$. Then
$$
\sum_{X^{\delta} < p^k \leq X} \frac{1}{p^k}|g(p^k)-A_g(X)+A_g(X/p^k)| \ll_{\alpha} (\log(1/\delta))^{1/2}\left(\frac{1}{X}\sum_{n \leq X} |g(n)-A_g(X)|^{\alpha}\right)^{1/\alpha} + \frac{B_g(X)}{(\log X)^{1/4}}.
$$
\end{lem}
\begin{proof}
To prove the result, we will estimate the quantity
$$
\mathscr{M} := \frac{1}{X}\sum_{X^{\delta} < p^k \leq X} \left|\sum_{\substack{n \leq X \\ p^k||n}} g(n) - \frac{1}{p^k}\left(1-\frac{1}{p}\right)\sum_{\ss{n \leq X}} g(n)\right|
$$
in two different ways. \\
First, using Lemma \ref{lem:mean} we observe that for each $p^k \leq X$,
\begin{align*}
\sum_{\substack{n \leq X \\ p^k||n}} g(n) &= \sum_{\ss{ mp^k\leq X \\ p \nmid m}} g(mp^k) = g(p^k)\left( \llf X/p^k\rrf - \llf X/p^{k+1}\rrf\right) + \sum_{\ss{m \leq X/p^k \\ p \nmid m}} g(m) \\
&= \frac{X}{p^k}\left(1-\frac{1}{p}\right) \left(g(p^k)+ A_g(X/p^k) - \sum_{\ss{p^j \leq X/p^k \\ j \geq 1}} \frac{g(p^j)}{p^j}\right) + O\left(|g(p^k)| + \frac{XB_g(X/p^k)}{p^k\sqrt{\log (2X/p^k)}}\right),
\end{align*}
where the second error term is $0$ unless $p^k \leq X/2$. Similarly, we have
$$
\frac{1}{p^k}\left(1-\frac{1}{p}\right) \sum_{\ss{n \leq X}} g(n) = \frac{X}{p^k}\left(1-\frac{1}{p}\right) \left(A_g(X) + O\left(\frac{B_g(X)}{\sqrt{\log X}}\right)\right).
$$
We thus deduce that
\begin{align} \label{eq:firstForm}
\mathscr{M} &= \sum_{X^{\delta} < p^k \leq X} \frac{1}{p^k}\left(1-\frac{1}{p}\right) \left|g(p^k) + A_g(X/p^k) - A_g(X)\right| \nonumber\\
&+ O\left(\sum_{X^{\delta} < p^k \leq X} \frac{1}{p^k} \sum_{p^j \leq X/p^k} \frac{|g(p^j)|}{p^j} + B_g(X) \sum_{\ss{X^{\delta} < p^k \leq X/2}} \frac{1}{p^k\sqrt{\log(X/p^k)}} + \frac{1}{X}\sum_{p^k \leq X} |g(p^k)| \right) \nonumber\\
&\geq \frac{1}{2}\sum_{X^{\delta} < p^k \leq X} \frac{1}{p^k} \left|g(p^k) + A_g(X/p^k) - A_g(X)\right| + O\left(\mc{R}(X) \right),
\end{align}
where we have set
$$
\mc{R}(X) :=\sum_{X^{\delta} < p^k \leq X} \frac{1}{p^k} \sum_{\ss{p^j \leq X/p^k \\ j \geq 1}} \frac{|g(p^j)|}{p^j} + B_g(X) \sum_{\ss{X^{\delta} < p^k \leq X \\ p^k \leq X/2}} \frac{1}{p^k\sqrt{\log(X/p^k)}} + \frac{1}{X}\sum_{p^k \leq X} |g(p^k)|.
$$
As in the proof of Lemma \ref{lem:mean}, the third expression in $\mc{R}(X)$ is
$$
\leq \frac{1}{\sqrt{X}} \sum_{p^k \leq X} \frac{|g(p^k)|}{p^{k/2}} \leq \left(\frac{\pi(X)}{X}\right)^{1/2} B_g(X) \ll \frac{B_g(X)}{\sqrt{\log X}}.
$$
Next, we may estimate the second expression in $\mc{R}(X)$ by
$$
\leq B_g(X)\left(\frac{1}{(\log X)^{1/4}} \sum_{\ss{X^{\delta} < p^k \leq Xe^{-\sqrt{\log X}}}}  \frac{1}{p^k} + \sum_{Xe^{-\sqrt{\log X}} < p^k \leq X/2} \frac{1}{p^k}\right) \ll \frac{B_g(X)}{(\log X)^{1/4}}.
$$
For the first expression in $\mc{R}(X)$, we use $|g(p^j)|/p^j \leq B_g(p^j)p^{-j/2}$ to get
\begin{align*}
&\sum_{X^{\delta} < p^k \leq X} \sum_{\ss{ p^j \leq X/p^k \\ j \geq 1 }} \frac{|g(p^j)|}{p^{j+k}} \leq B_g(X) \sum_{X^{\delta} < p^k \leq X} \sum_{ \ss{ p^j \leq X/p^k \\ j \geq 1}} p^{-(k+j/2)} \\
&\ll B_g(X) \left(X^{-\delta} \sum_{\ss{X^{\delta} < p^k \leq X \\ k \geq 2/\delta}} 1 + \sum_{\ss{X^{\delta} < p^k \leq X \\ p > X^{\delta^2/2}}}  p^{-k} \sum_{\ss{p^j \leq X/p^k \\ j \geq 1}} \frac{1}{p^{j/2}}\right) \\
&\ll B_g(X)\left(X^{-\delta/2} + \sum_{p > X^{\delta^2/2}} p^{-3/2}\right) \ll B_g(X) X^{-\delta^2/4}.
\end{align*}
Thus, we obtain $\mc{R}(X) \ll B_g(X)/(\log X)^{1/4}$, and so
$$
\sum_{X^{\delta} < p^k \leq X} \frac{1}{p^k} |g(p^k)+A_g(X/p^k) - A_g(X)| \leq 2 \mathscr{M} + O(B_g(X)/(\log X)^{1/4}).
$$
Next, we execute the second estimation of $\mathscr{M}$.
If $p^k \in (X^{\delta},X]$ define
$$
\Delta_{g}(X;p^k) := \frac{p^k}{X}\left|\sum_{\substack{n \leq X \\ p^k ||n}} g(n) - \frac{1}{p^k}\left(1-\frac{1}{p}\right)\sum_{\ss{n \leq X}} g(n)\right|.
$$
Set $g'(n) := g(n) -A_g(X)$ for $n \leq X$, and note that
$$
\Delta_g(X;p^k) = \Delta_{g'}(X;p^k) + O\left(\frac{p^k}{X}|A_g(X)|\right).
$$
Thus, we find
\begin{equation}\label{eq:interim}
\mathscr{M} = \sum_{X^{\delta} < p^k \leq X} \frac{1}{p^k}\Delta_{g'}(X;p^k)+ O\left(\frac{|A_g(X)|\pi(X)}{X}\right) = \sum_{X^{\delta}< p^k \leq X} \frac{1}{p^k}\Delta_{g'}(X;p^k) + O\left(\frac{B_g(X) \log\log X}{\log X}\right).
\end{equation}
Let us now partition the set of prime powers $X^{\delta} < p^k \leq X$ into the sets
\begin{align*}
\mc{P}_1 &:= \left\{X^{\delta} < p^k \leq X : \Delta_{g'}(X;p^k) > \left(\frac{1}{X}\sum_{n \leq X}|g'(n)|^{\alpha}\right)^{1/\alpha}\right\} \\
\mc{P}_2 &:= \left\{X^{\delta} < p^k \leq X : \Delta_{g'}(X;p^k) \leq \left(\frac{1}{X}\sum_{n \leq X}|g'(n)|^{\alpha}\right)^{1/\alpha}\right\}.
\end{align*}
Note that by Mertens' theorem,
\begin{align} \label{eq:ppMertens}
\sum_{X^{\delta} < p^k \leq X} \frac{1}{p^k} &\leq \sum_{X^{\delta} < p \leq X} \frac{1}{p} + X^{-\delta} \sum_{\ss{X^{\delta} < p^k \leq X \\ k \geq 2/\delta}} 1 + \sum_{\ss{X^{\delta} < p^k \leq X \\ 2 \leq k \leq 2/\delta \\ p > X^{\delta^2/2}}} \frac{1}{p^k} \nonumber\\
&\ll \log(1/\delta)\left(1+X^{-\delta^2/2}\right) + X^{-\delta/2} \ll \log(1/\delta).
\end{align}
Using this and H\"{o}lder's inequality, we obtain
\begin{align*}
&\sum_{X^{\delta} < p^k \leq X} p^{-k} \Delta_{g'}(X;p^k) \\
&\ll_{\alpha} \left(\log\left(1/\delta\right)\right)^{1-\frac{1}{\alpha}}\left(\sum_{\substack{X^{\delta} < p^k \leq X \\ p^k \in \mc{P}_1}} p^{-k} \Delta_{g'}(X;p^k)^{\alpha}\right)^{\frac{1}{\alpha}} + (\log(1/\delta))^{\frac{1}{2}} \left(\sum_{\substack{X^{\delta} < p^k \leq X \\ p^k \in \mc{P}_2}} p^{-k} \Delta_{g'}(X;p^k)^{2}\right)^{\frac{1}{2}}.
\end{align*}
By Theorem 3.1 of \cite{EllDual} this is bounded by
$$
\ll (\log(1/\delta))^{\frac 12}\left(\frac{1}{X} \sum_{n \leq X}|g'(n)|^{\alpha}\right)^{\frac 1\alpha}.
$$
Combining this with \eqref{eq:firstForm} and \eqref{eq:interim} completes the proof of the lemma. 
\end{proof}

Next, we show that, in an $\ell^1$ sense, $g(p^k)$ is well-approximated by $\lambda \log p^k$ on average over the prime powers $X^{\delta} < p^k \leq X$, for some function $\lambda = \lambda(X)$.
\begin{lem} \label{lem:AlphaMom}
There is a parameter $\lambda=\lambda(X) \in \mb{R}$ such that the following holds. For any $\alpha \in (1,2)$, 
$$
\sum_{X^{\delta}<p^k \leq X} \frac{1}{p^k}|g(p^k)-\lambda \log p^k| \ll_{\alpha} (\log(1/\delta))^{1/2}\left(\frac{1}{X}\sum_{n \leq X} |g(n)-A_g(X)|^{\alpha}\right)^{1/\alpha}.
$$
\end{lem}

\begin{proof}
By \cite[Th\'{e}or\`{e}me 1]{HilMom}, there is $\lambda = \lambda(X)$ and $c = c(X)$, depending on $g$ but independent of $\alpha$, such that
\begin{align}\label{eq:momentApp}
\left(\frac{1}{X}\sum_{n \leq X} |g(n)-A_g(X)|^{\alpha}\right)^{\alpha} \gg_{\alpha} \left(\sum_{p^k \leq X} \frac{|g_{\lambda,c}''(p^k)|^{\alpha}}{p^k}\right)^{1/\alpha} + \left(\sum_{p^k \leq X} \frac{|g_{\lambda,c}'(p^k)|^2}{p^k}\right)^{1/2}.
\end{align}
Here, writing $g_{\lambda}(n) = g(n)-\lambda \log(n)$, we have set
$$
g_{\lambda,c}'(p^k) := \begin{cases} g_{\lambda}(p^k) &\text{ if } |g_{\lambda}(p^k)| \leq c \\ 0 &\text{ otherwise;} \end{cases} \quad \quad g_{\lambda,c}''(p^k) := \begin{cases} 0 &\text{ if } |g_{\lambda}(p^k)| \leq c \\ g_{\lambda}(p^k) &\text{ otherwise.} \end{cases}
$$
Since $g_{\lambda}(p^k) = g_{\lambda,c}'(p^k) + g_{\lambda,c}''(p^k)$ for each $p^k$, by H\"{o}lder's inequality and \eqref{eq:ppMertens} once again,
$$
\sum_{X^{\delta} < p^k \leq X} \frac{|g_{\lambda}(p^k)|}{p^k} \ll_{\alpha} (\log(1/\delta))^{1-1/\alpha}\left(\sum_{X^{\delta} < p^k \leq X} \frac{|g_{\lambda,c}''(p^k)|^{\alpha}}{p^k}\right)^{1/\alpha} + (\log(1/\delta))^{1/2}\left(\sum_{X^{\delta} < p^k \leq X} \frac{|g_{\lambda,c}'(p^k)|^2}{p^k}\right)^{1/2}.
$$
The results now follows upon combining this last estimate with \eqref{eq:momentApp} and using positivity.
\end{proof}

\begin{lem} \label{lem:SmallMoments}
Assume that $\mc{B} := \{n \in \mb{N} : g(n) < g(n-1)\}$ satisfies $d\mc{B} = 0$. Then there is an absolute constant $c > 0$ for any $\alpha \in [1,2)$,
$$
\frac{1}{X} \sum_{n \leq X} |g(n)-A_g(X)|^{\alpha} \ll \left(\frac{\log\log(1/r(X))}{\log(1/r(X))} + (\log X)^{-c} \right)^{2-\alpha} B_g(X)^{\alpha},
$$
where $r(x) := \left(\frac{|\mc{B}(X)|}{X}\right)^{1/2} + \frac{\log X}{\sqrt{X}}$.
\end{lem}

\begin{proof}
By H\"{o}lder's inequality, for any $\alpha \in (1,2)$,
we have
$$
\frac{1}{X}\sum_{n \leq X} |g(n)-A_g(X)|^{\alpha} \leq \left(\frac{1}{X}\sum_{n \leq X}|g(n)-A_g(X)|\right)^{2-\alpha} \cdot \left(\frac{1}{X}\sum_{n \leq X} |g(n)-A_g(X)|^2\right)^{\alpha-1},
$$
an inequality that is vacuously also true when $\alpha = 1$. Applying Lemma \ref{lem:TK} to the second bracketed expression, we obtain the upper bound
$$
\ll B_g(X)^{2\alpha-2} \left(\frac{1}{X}\sum_{n \leq X} |g(n)-A_g(X)|\right)^{2-\alpha}.
$$
Next, we show that
\begin{equation}\label{eq:suffGaps}
\frac{1}{X}\sum_{n \leq X}|g(n)-g(n-1)| \ll \left(\left(\frac{|\mc{B}(X)|}{X}\right)^{1/2} + \frac{\log X}{\sqrt{X}} \right)B_g(X) = r(X) B_g(X).
\end{equation}
This will imply the claim of the lemma, since by Proposition \ref{prop:convGaps} the latter bound implies that for some small constant $c > 0$,
$$
\frac{1}{X}\sum_{n \leq X} |g(n)-A_g(X)| \ll \left(\frac{\log\log (1/r(X))}{\log(1/r(X))} + (\log X)^{-c}\right) B_g(X).
$$
To prove \eqref{eq:suffGaps}, we note that for all $1 \leq n \leq X$, 
$$
|g(n)-g(n-1)| = g(n)-g(n-1) + 2|g(n)-g(n-1)|1_{n \in \mc{B}(X)}.
$$
It follows from this and telescoping that 
$$
\frac{1}{X}\sum_{n \leq X} |g(n)-g(n-1)| = \frac{g(\llf X\rrf)}{X} + \frac{2}{X}\sum_{n \in \mc{B}(X)} |g(n)-g(n-1)|.
$$
By Lemma \ref{lem:ptwiseBdg}, $g(\llf X \rrf)/X \ll B_g(X)(\log X)/\sqrt{X}$. Owing to Lemma \ref{lem:TK} and the triangle and Cauchy-Schwarz inequalities, we also obtain
$$
\frac{1}{X}\sum_{n \in \mc{B}(X)}|g(n)-g(n-1)| \leq 2\left(\frac{|\mc{B}(X)|}{X}\right)^{1/2} \left(\frac{1}{X} \sum_{n \leq X} |g(n)-A_g(X)|^2\right)^{1/2} \ll B_g(X) \left(\frac{|\mc{B}(X)|}{X}\right)^{1/2}.
$$
This implies \eqref{eq:suffGaps}, and completes the proof of the lemma.
\end{proof}

\begin{proof}[Proof of Proposition \ref{prop:AgLog}]
The first part of the proposition follows by the triangle inequality, upon combining Lemmas \ref{lem:EllLS} and \ref{lem:AlphaMom} with Lemma \ref{lem:SmallMoments} (taking $\alpha = 3/2$, say). \\
Next, we proceed to the proofs of properties (i)-(iii). \\
(i) By the triangle inequality and positivity, we obtain
$$
\lambda(X)\sum_{X^{1/4} < p \leq X^{1/2}} \frac{\log p}{p} \leq
\sum_{X^{1/4} < p \leq X} \frac{|A_g(X)-A_g(X/p)-\lambda(X)\log p|}{p} + \sum_{X^{1/4} < p \leq X^{1/2}} \frac{|A_g(X)-A_g(X/p)|}{p}.
$$
By Mertens' theorem, 
\begin{align*}
&\sum_{X^{1/4} < p \leq X^{1/2}} \frac{\log p}{p} = \frac{1}{4}\log X+ O\left(\frac{1}{\log X}\right) \gg \log X,
\end{align*}
and by the Cauchy-Schwarz inequality we have, for $X^{1/4} \leq p \leq X^{1/2}$,
\begin{align*}
&|A_g(X)-A_g(X/p)| \ll B_g(X) \left(\sum_{X/p \leq q^k \leq X} \frac{1}{q^k}\right)^{1/2} \ll B_g(X)\left(\frac{\log p}{\log X}\right)^{1/2}.
\end{align*}
We thus deduce from the first part of the proposition, the prime number theorem and partial summation that
$$
\lambda(X) \log X \ll o(B_g(X)) + \frac{B_g(X)}{\sqrt{\log X}} \sum_{p \leq X^{1/2}} \frac{(\log p)^{1/2}}{p} \ll B_g(X),
$$
and (i) follows immediately.\\
(ii) We observe, using (i) and Lemmas \ref{lem:AlphaMom} and \ref{lem:SmallMoments} that if $X^{\delta} < t_1 \leq t_2 \leq X$,
\begin{align*}
&\left|A_g(t_2)-A_g(t_1) - \lambda(X) \log(t_2/t_1)\right| \\
&= \left|\sum_{t_1 < p^k \leq t_2} \left(1-\frac{1}{p}\right) \frac{g(p^k)-\lambda(X) \log p^k}{p^k}\right| + O\left(\lambda(X) \left(\frac{\delta^{-1}}{\log X} + \sum_{t_1 < p^k \leq t_2} \frac{\log p^k}{p^{k+1}} \right) \right) \\
&\leq \sum_{X^{\delta} < p^k \leq X} \frac{|g(p^k)-\lambda(X)\log p^k|}{p^k} + O\left(\frac{B_g(X)}{\sqrt{\log X}}\right) \\
&= o((\log(1/\delta)^{1/2}B_g(X)),
\end{align*}
as required.\\
(iii) Applying (ii) with $(t_1,t_2) = (X^y,X^z)$, where $(y,z) = (u,1)$, $(y,z) = (uv,1)$ and $(y,z) = (uv,u)$ for any $v \in (\delta/u,1/2]$ and $u \in (\delta,1]$, we get
\begin{align*}
A_g(X) - A_g(X^u) &= (1-u) \lambda(X)\log X + o((\log(1/\delta)^{1/2}B_g(X)) \\
A_g(X)-A_g(X^{uv}) &= (1-uv)\lambda(X)\log X + o((\log(1/\delta)^{1/2}B_g(X)) \\
A_g(X^u) - A_g(X^{uv}) &= (u-uv) \lambda(X^u) \log X + o((\log(1/\delta)^{1/2}B_g(X^u)).
\end{align*}
Combining these equations and using $B_g(X^u) \leq B_g(X)$, we conclude that
$$
u(1-v) \lambda(X)\log X = u(1-v)\lambda(X^u)\log X + o((\log(1/\delta)^{1/2}B_g(X)).
$$
Since $1-v \geq 1/2$ and $u > \delta$, the claim follows immediately upon rearranging (with a potentially larger implicit constant in the error term).
\end{proof}

\begin{proof}[Proof of Theorem \ref{thm:iterStep}: Part II]
The first part of the proof implied that
$$
\sum_{X^{1/2} < p^k \leq X} \frac{|g(p^k)-\lambda_0(X)\log p^k|^2}{p^k} = o(B_g(X)^2).
$$
Now, combining Lemmas \ref{lem:AlphaMom} and \ref{lem:SmallMoments}, we have that
$$
\sum_{X^{\delta} < p^k \leq X} \frac{|g(p^k)-\lambda(X) \log p^k|}{p^k} = o(B_g(X)),
$$
where, according to Proposition \ref{prop:AgLog} we have
$$
\lambda(X) = \lambda(X^u) + O(B_g(X)/\log X), \quad 0 < u \leq 1 \text{ fixed.}
$$
Thus, by these two estimates, Cauchy-Schwarz and Mertens' theorem, whenever $Y = X^u$ with $0 < u \leq 1$ fixed, we have
\begin{align}
|\lambda(Y)-\lambda_0(Y)| \log Y &\ll |\lambda(Y)-\lambda_0(Y)| \sum_{Y^{1/2} < p \leq Y} \frac{\log p}{p} \\
&\leq \sum_{Y^{1/2} < p \leq Y} \frac{|g(p)-\lambda_0(Y)\log p|}{p} + \sum_{Y^{1/2} < p \leq Y} \frac{|g(p)-\lambda(Y)\log p|}{p} \nonumber\\
&\leq B_{g_{\lambda_0}}(Y) \left(\sum_{Y^{1/2} < p \leq Y} \frac{1}{p}\right)^{1/2} + o(B_g(Y)) = o(B_g(X)). \label{eq:compLambLamb0}
\end{align}
We thus deduce that $\lambda(X^u) = \lambda_0(X^u) + o(B_g(X)/\log X)$ for all $0 < u \leq 1$ fixed, and therefore also that
$$
\lambda_0(X^u) = \lambda(X^u) + o\left(\frac{B_g(X)}{\log X}\right) = \lambda(X) + o\left(\frac{B_g(X)}{\log X}\right) = \lambda_0(X) + o\left(\frac{B_g(X)}{\log X}\right),
$$
and the second claim of Theorem \ref{thm:iterStep} is proved.
\end{proof}

\subsection{Proof of Corollary \ref{cor:weakErd}} \label{subsec:pfWeakErd}
In this subsection we prove Corollary \ref{cor:weakErd}. The key step will be to show that if there is a $\lambda(X)$ such that $B_{g_{\lambda}}(X) = o(B_g(X))$ (which follows from Theorem \ref{thm:iterStep}) then $B_g(X)$ grows like $\log X$. \\
We begin by showing that if $\lambda(X)$ is fairly large then $B_g(X)$ is close to $\log X$.
\begin{lem} \label{lem:largeLambda}
Assume that there is a $C > 0$ such that $\lambda(X) \geq C B_g(X)/\log X$ for all $X$ sufficiently large in the conclusion of Proposition \ref{prop:AgLog}. Then for any $\e > 0$, $(\log X)^{1-\e} \ll_{\e} B_g(X) \ll_{\e} (\log X)^{1+\e}$.
\end{lem}
\begin{proof}
By Proposition \ref{prop:AgLog}, 
\begin{equation}\label{eq:slowvar}
\lambda(X) = \lambda(X^u) + o(B_g(X)/\log X) = \lambda(X^u) + o(\lambda(X))
\end{equation}
whenever $0 < u \leq 1$ is fixed. This implies in particular that $\lambda(X) \ll \lambda(X^u)$. Setting $Y := X^u$ and $v := 1/u \geq 1$, we see also that 
$$
\lambda(Y^v) = \lambda(Y) + o(\lambda(Y^v)) = \lambda(Y) + o(\lambda(Y^{uv})) =\lambda (Y) + o(\lambda(Y)).
$$
Thus, \eqref{eq:slowvar} holds for all fixed $u \geq 1$ as well, and thus for all $u > 0$. We thus deduce that for each $u > 0$ fixed and $\e > 0$ there is $X_0(\e,u)$ such that if $X\geq X_0(\e,u)$,
$$
\left|\frac{\lambda(X^u)}{\lambda(X)} - 1\right| < \e.
$$
Set $u = 1/2$, put $X_0 = X_0(\e,1/2)$ and for each $k \geq 1$ define $X_k := X_0^{2^k}$. Then we have
$$
\frac{\lambda(X_0)}{\lambda(X_K)} = \prod_{1 \leq k \leq K} \frac{\lambda(X_{k-1})}{\lambda(X_k)} \in [(1-\e)^K, (1+\e)^K].
$$
As $K \leq 2\log\log X_K$, we find that for $K$ large enough,
\begin{align*}
|\lambda(X_K)| &\leq |\lambda(X_0)| \exp\left(-K\log(1-\e)\right) \ll_{\e} \exp\left(4\e \log\log X_K\right) = (\log X_K)^{4\e}, \\
|\lambda(X_K)| &\geq |\lambda(X_0)| \exp\left(-K\log(1+\e)\right) \gg_{\e} \exp\left(-4\e \log\log X_K\right) = (\log X_K)^{-4\e}.
\end{align*}
Thus, we have $B_g(X_K) \ll \lambda(X_K) (\log X_K) \ll_{\e} (\log X_K)^{1+4\e}$ by assumption, and by Proposition \ref{prop:AgLog}(i) we have $B_g(X_K) \gg |\lambda(X_K)| \log X_K \gg_{\e} (\log X_k)^{1-4\e}$. \\
Since $\log X_K \asymp \log X_{K+1}$, by monotonicity we also have
$$
B_g(X) \leq B_g(X_{K+1}) \ll_{\e} (\log X_{K+1})^{1+4\e} \ll (\log X_K)^{1+4\e} \leq (\log X)^{1+4\e}
$$
for any $X_K < X < X_{K+1}$. Similarly, we also obtain $B_g(X) \gg_{\e} (\log X)^{1-4\e}$ on the same interval. Since $\e > 0$ was arbitrary, the claim now follows.
\end{proof}

\begin{lem} \label{lem:smallBgLambda}
Assume $B_{g_{\lambda_0}}(X) = o(B_g(X))$ for some $\lambda_0 = \lambda_0(X)$ that satisfies $|\lambda_0| \ll B_g(X)/\log X$. 
Then for any $\e > 0$, $(\log X)^{1-\e} \ll_{\e} B_g(X) \ll_{\e} (\log X)^{1+\e}$.
\end{lem}
\begin{proof}
By Cauchy-Schwarz, we have
$$
B_g(X)^2 = \sum_{p^k \leq X} \frac{|g(p^k)|^2}{p^k} \leq 2\left(\lambda_0(X)^2 \sum_{p^k \leq X} \frac{(\log p^k)^2}{p^k} + \sum_{p^k \leq X} \frac{|g_{\lambda_0}(p^k)|^2}{p^k}\right) = \lambda_0(X)^2 (\log X)^2 + o(B_g(X)^2).
$$
It follows that $|\lambda_0(X)| \geq \frac{1}{2}B_g(X)/\log X$ when $X$ is sufficiently large. The conclusion follows from Lemma \ref{lem:largeLambda}, provided we can show that $\lambda(X) = \lambda_0(X) + o(B_g(X)/\log X)$ for all large $X$, where $\lambda(X)$ is the function from the conclusion of Proposition \ref{prop:AgLog}. 
But this was verified in \eqref{eq:compLambLamb0},
so the claim follows.
\end{proof}

\begin{proof}[Proof of Corollary \ref{cor:weakErd}]
Suppose $g: \mb{N} \ra \mb{R}$ is a completely additive function that satisfies 
$$
F_g(\e) \ra 0 \text{ as } \e \ra 0^+, \text{ and } |\mc{B}(X)| \leq \frac{X}{(\log X)^{2+\eta}}, \text{ for some } \eta> 0.
$$
Suppose first that $B_g(X) \ra \infty$, so that $g \in \mc{A}_s$. By Theorem \ref{thm:iterStep} there is a parameter $\lambda_0(X)$ with $|\lambda_0(X)| \ll B_g(X)/\log X$, such that $B_{g_{\lambda_0}}(X) = o(B_g(X))$, as $X \ra \infty$. By Lemma \ref{lem:smallBgLambda}, we deduce that $B_g(X) \ll_{\e} (\log X)^{1+\e}$. Now, applying \eqref{eq:suffGaps}, we obtain
\begin{align*}
\frac{1}{X}\sum_{n \leq X} |g(n)-g(n-1)| &\ll B_g(X) \left(\left(\frac{|\mc{B}(X)|}{X}\right)^{\frac 12} + \frac{\log X}{\sqrt{X}}\right) \ll (\log X)^{1+\frac{\eta}{3}} \cdot (\log X)^{-\frac{1}{2}(2+\eta)} \ll (\log X)^{-\frac{\eta}{6}}.
\end{align*}
By Theorem \ref{thm:KatWir}, we deduce that there is a constant $c\in \mb{R}$ such that $g(n) = c\log n$ for all $n$, as required.\\
If, instead, $B_g(X) \ll 1$ then we again deduce (even if $g \notin \mc{A}_s$) from \eqref{eq:suffGaps} that
$$
\frac{1}{X}\sum_{n \leq X} |g(n)-g(n-1)| = o(B_g(X)) = o(1),
$$
and so the claim follows (necessarily with $c = 0$) by Theorem \ref{thm:KatWir}.
\end{proof}

\subsection{Proof of Theorem \ref{thm:almErd}}
In \cite{EllFA}, Elliott shows the following.
\begin{thm1}[\cite{EllFA}, Thm. 6] \label{thm:ThmEllFA}
Let $0 < a < b \leq 1$. Let $g: \mb{N} \ra \mb{C}$ be an additive function, and for $y \geq 10$ define
$$
\theta(y) := \sum_{y^{a} < p^k \leq y^{b}} \frac{1}{p^k} |g(p^k)-A_g(y) + A_g(y/p^k)|.
$$
Then for all $\e, B > 0$ there exist $X_0 = X_0(a,b,\e,B)$ and $c > 0$ such that if $X \geq X_0$ then, uniformly over $X^{\e} < t \leq X$,
$$
A_g(t) = G(X) \log t - \eta(X) + O(Y(X)),
$$
where $G,\eta$ are measurable functions and
$$
Y(X) := \sup_{X^c < w \leq X} \theta(w) + (\log X)^{-B} \sum_{p^k \leq X} \frac{|g(p^k)|}{p^k} + \max_{X^c \leq p^k \leq X} |g(p^k)|p^{-k}.
$$
\end{thm1}

\begin{cor} \label{cor:EllFA}
Let $\delta \in (0,1/2)$. Suppose $g: \mb{N} \ra \mb{R}$ is an additive function such that $|\mc{B}(X)| = o(X)$. Then, uniformly over all $X^{\delta} \leq t \leq X$ we have
$$
A_g(t) = \lambda(X)\log t - \eta(X) + o(B_g(X)),
$$
where $\lambda(X)$ and $\eta(X)$ are measurable functions such that for each fixed $0 < u \leq 1$,
$$
\lambda(X^u) = \lambda(X) + o(B_g(X)/\log X), \quad \quad \eta(X^u) = \eta(X) + o(B_g(X)).
$$
\end{cor}
\begin{proof}
By combining Lemmas \ref{lem:EllLS} and \ref{lem:SmallMoments}, we have
$$
\sum_{X^{\delta} \leq p^k \leq X} \frac{|g(p^k)-A_g(X) + A_g(X/p^k)|}{p^k} = o(B_g(X)),
$$
for any fixed $\delta > 0$. Applying Elliott's theorem with $a = \e = \delta$, $b = 1$, $B = 1$, we have, uniformly over all $X^c < y \leq X$ 
$$
Y(X) = o(B_g(X)) + O\left(\frac{\sqrt{\log\log X}}{\log X} B_g(X) + B_g(X)X^{-c/2}\right)  = o(B_g(X)),
$$
using the bound $|g(p^k)|p^{-k/2} \leq B_g(X)$ for all $p^k \leq X$. We thus deduce that (relabeling $G$ as $\lambda$)
\begin{equation}\label{eq:AgExpHere}
A_g(t) = \lambda(X) \log t - \eta(X) + o(B_g(X)).
\end{equation}
We may observe, analogously as in the proof of Proposition \ref{prop:AgLog} that $\lambda(X^u) = \lambda(X) + o(B_g(X)/\log X)$ for fixed $u \in (\delta,1)$. Furthermore, evaluating $A_g(X^u)$ in \eqref{eq:AgExpHere}, once as written and once with $X$ replaced by $X^u$, we obtain that
$$
A_g(X^u) = u\lambda(X^u) \log X -\eta(X^u) + o(B_g(X^u)) = u\lambda(X) \log X - \eta(X) + o(B_g(X)),
$$
from which it also follows, using the slow variation of $\lambda$, that
$$
\eta(X^u) = \eta(X) + u(\lambda(X^u) - \lambda(X)) \log X + o(B_g(X)) = \eta(X) + o(B_g(X))
$$
for each fixed $\delta \leq u \leq 1$, as required.
\end{proof}

\begin{proof}[Proof of Theorem \ref{thm:almErd}]
By Lemma \ref{lem:SmallMoments} (with $\alpha = 1$),
$$
\frac{1}{X}\sum_{n \leq X} |g(n)-A_g(X)| = o(B_g(X))
$$
so that for all but $o(X)$ integers $n \leq X$ we have
\begin{equation}\label{eq:conc}
g(n) = A_g(X)+ o(B_g(X)).
\end{equation}
By Corollary \ref{cor:EllFA} and Proposition \ref{prop:AgLog}(i), we deduce that
$$
g(n) = \lambda(X) \log X - \eta(X) + o(B_g(X)) = \lambda(X) \log n - \eta(X) + o(B_g(X))
$$
for all but $o(X)$ integers $X/\log X < n \leq X$, and thus for all but $o(X)$ integers $n \leq X$, proving the claim of Theorem \ref{thm:almErd}.
%
\end{proof}

\section*{Acknowledgments}
The author warmly thanks 
Oleksiy Klurman 
and Aled Walker 
for helpful suggestions about improving the exposition of the paper, as well as for their encouragement. Most of this paper was written while the author held a Junior Fellowship at the Mittag-Leffler institute for mathematical research during the Winter of 2021. He would like to thank the institute for its support. 

\bibliography{ErdConj}
\bibliographystyle{plain}
\end{document}